\documentclass[pdflatex,sn-mathphys]{sn-jnl}
\usepackage[all,cmtip]{xy}

\jyear{2022}%

\theoremstyle{plain} 
\newtheorem{thm}{Theorem}[section] 
\newtheorem{cor}[thm]{Corollary} 
\newtheorem{lem}[thm]{Lemma} 
\newtheorem{prop}[thm]{Proposition}

\theoremstyle{definition} 
\newtheorem{defn}{Definition}[section] 
\newtheorem{exem}{Example} [section]

\theoremstyle{remark} 
\newtheorem{rem}{Remark}

\begin{document}

\title[]{$L^2$-cohomology and quasi-isometries on the ends of unbounded geometry}


\author*[1,2]{\pfx{Dr} \fnm{Stefano} \sur{Spessato} \sfx{ORCID: 0000-0003-0069-5853}\email{stefano.spessato@unicusano.it }}

\affil*[1]{\orgname{Università degli Studi "Niccolò Cusano"}, \orgaddress{\street{via don Carlo Gnocchi 3}, \city{Rome}, \postcode{00166}, \country{Italy}}}


\abstract{In this paper we study the minimal and maximal $L^{2}$-cohomology of oriented, possibly not complete, Riemannian manifolds. Our focus will be on both the reduced and the unreduced $L^{2}$-cohomology groups. In particular we will prove that these groups are invariant under uniform homotopy equivalence quasi-isometric on the unbounded ends. A \textit{uniform map} is a uniformly continuous map such that the diameter of the preimage of a subset is bounded in terms of the diameter of the subset itself. A map $f$ between two Riemannian manifolds $(X,g)$ and $(Y,h)$ is \textit{quasi-isometric on the unbounded ends} if $X = M \cup E_X$ where $M$ is the interior of a manifold of bounded geometry with boundary, $E_X$ is an open of $X$ and the restriction of $f$ to $E_X$ is a quasi-isometry. Finally some consequences are shown: the main ones are definition of a mapping cone for $L^2$-cohomology and the invariance of the $L^2$-signature.}




\maketitle
\subsection*{Acknowledgments}
I would like to thank Francesco Bei for his suggestions and for his very helpful mails.
\bigskip

Data sharing not applicable to this article as no datasets were generated or analysed during the current study.

\section*{Introduction}
In this paper a generalization of the work done by the author in \cite{Spes} is proved. In that work, given a uniform map $f:(M,g) \longrightarrow (N,h)$ between manifolds of bounded geometry, an $\mathcal{L}^2$-bounded operator $T_f: \mathcal{L}^2(N,h) \longrightarrow \mathcal{L}^2(M,g)$ between the spaces of squared-integrable forms was introduced. A \textit{uniform} map is a uniformly continuous map such that for each compact subset $A$, the diameter of the preimage of $A$ is bounded in terms of the diameter of $A$. The operator $T_f$ induces an operator between the reduced and unreduced $L^2$-cohomology groups and it replaces the pullback operator which is not a well-defined operator between the $\mathcal{L}^2$-spaces. As a consequence of this we obtain that the reduced and the unreduced $L^2$-cohomology are invariant under uniform homotopy equivalences.
\\Our goal in this paper is to prove a similar result for the minimal and maximal $L^2$-cohomologies, in both the reduced and unreduced version, of possibly not complete Riemannian manifolds. In order to obtain such a result we need some further assumptions on the homotopy equivalences. Let us briefly introduce these assumptions.
\bigskip
\\Given a Riemannian manifold $(X,g)$, it is always possible to decompose it as $X = M \cup E_X$ where $E_X$ is an open subset of $X$ and $M$ is the interior of a manifold with boundary of bounded geometry (a definition can be found in the paper of Schick \cite{Schick}). We will say that $M$ is an \textit{open of bounded geometry} and $E_X$ will be the \textit{unbounded ends} of $X$. A manifold $(X,g)$ decomposed in this way is called a \textit{manifold of bounded geometry with unbounded ends}.
\\In this paper we will study uniform maps $f:(X,g) \longrightarrow (Y,h)$ between oriented manifolds which are \textit{quasi-isometric on unbounded ends}, i.e. given two manifolds of bounded geometry with unbounded ends $(X = M \cup E_X,g)$ and $(Y = N \cup E_Y,h)$, then $f(E_X) \subseteq E_Y$ and $f_{\vert E_X}$ is a quasi-isometry.
\\So we want to show that if $f: (X,g) \longrightarrow (Y,h)$ is a uniform homotopy equivalence which is quasi-isometric on the unbounded ends, then 
\begin{equation}
    \begin{cases}
    H^k_{2, max}(X,g) \cong H^k_{2, max}(Y,h) \\
    \overline{H}^k_{2, min}(X,g) \cong H^k_{2, min}(Y,h) \\
    H^k_{2, min}(X,g) \cong H^k_{2, min}(Y,h) \\
    \overline{H}^k_{2, min}(X,g) \cong \overline{H}^k_{2, min}(Y,h)
    \end{cases}
\end{equation}
where $H^k_{2, min \backslash max}(X,g)$ and $H^k_{2, min \backslash max}(Y,h)$ are the minimal and maximal $k$-th group of $L^2$-cohomology and $\overline{H}^k_{2, min\backslash max}(X,g)$ and $\overline{H}^k_{2, min\backslash max}(Y,h)$ are their reduced versions. This result is proved in Proposition \ref{proposition1}. Finally we will show some consequences.
\bigskip
\\The structure of the paper is the following: in the first section we introduce the notions of uniform maps quasi-isometric on the unbounded ends and some examples are showed. In the second section we introduce the minimal and maximal $L^2$-cohomology of a Riemannian manifold. In third section we recall the notion of Radon-Nikodym-Lipschitz map. These are maps such that their pullback induced a well-defined $\mathcal{L}^2$-bounded operator. In particular we recall a technique from \cite{Spes} which allows to know if a Lipschitz submersion is a Radon-Nikodym-Lipschitz map or not.
\\In the fourth section, given a uniform map isometric on the unbounded ends $f:(X,g) \longrightarrow (Y,h)$ between oriented manifolds, we introduce a new version of the operator $T_f$. In particular if $(X,g)$ and $(Y,h)$ are manifolds of bounded geometry, then $T_f$ is exactly the same defined in \cite{Spes}. In the fifth section, Proposition \ref{proposition1}, the main result stated above is proved.
\\In the last section we will show some consequences of the existence of $T_f$: first we define a mapping cone for minimal and maximal $L^2$-cohomology. Then we will show that the $L^2$-signature of a complete manifold of dimension $4k$ is invariant under uniform homotopy equivalences isometric on the ends which preserve the orientations. We also prove a result similar to point 1 of  Proposition 5 of the work of Lott \cite{Lott}: if two homotopy equivalent manifolds are isometric outside two compact subsets, then their maximal and minimal $L^2$-cohomology are equivalent. The same also holds for their reduced versions.
\\Finally an example of Riemannian manifolds with the same minimal and maximal $L^2$-cohomology (reduced and not) which are not quasi-isometric is shown.

\section{Quasi-isometries on the unbounded ends and open subsets of bounded geometry}
In this section we introduce the geometric setting of our paper. In particular we introduce the notions of \textit{open of bounded geometry} and of \textit{uniform homotopy quasi-isometric on the unbounded ends}.
\subsection{Manifold of bounded geometry with some unbounded ends}
The following definition is the Definition 2.2 of the work of Schick \cite{Schick}.
\begin{defn}\label{open}
Let $(\overline{M},g)$ an oriented Riemannian manifold with boundary $\partial \overline{M}$ (possibily empty) and denote by $M$ the interior of $\overline{M}$. Consider on $\partial \overline{M}$ the Riemannian metric induced by $g$. Let us denote by $l$ the second fundamental form of $\partial M$ and by $\overline{\nabla}$ the Levi-Civita connection on $\partial \overline{M}$. Then we say that $(\overline{M},g)$ is a \textbf{manifold with boundary and of bounded geometry} if the following holds.
\begin{enumerate}
    \item there exists a number $r_C > 0$ such that the \textit{geodesic collar}
    \begin{equation}
    \begin{split}
    \partial \overline{M} \times [0,r_c) &\longrightarrow M \\
    (x,t) &\longrightarrow exp_x(t\nu_x)
    \end{split}
    \end{equation}
    is a diffeomorphism with its image ($\nu_x$ is the unit inward normal vector).
    \item The injectivity radius $inj_{\partial \overline{M}}$ of $\partial \overline{M}$ is positive.
    \item Consider a number $r$ and define $\mathcal{N}(r) := B_r(\partial \overline{M}) \cap M$. There is a $r_i > 0$ such that if $r \leq r_i$ we have that for each $p$ in $M \setminus \mathcal{N}(r)$ the exponential map is a diffeomorphism on $B_{0_p}(r) \subset T_p M$.
    \item For every $k \in \mathbb{N}$ there is a constant $C_k $ so that $\vert \nabla^i R \vert \leq C_i$ and $\vert \overline{\nabla}^il \vert \leq C_i$ for each $i = 0, \ldots, k$.
\end{enumerate}
Finally, if $M$ is the interior of a manifold of bounded geometry, then it is an \textbf{open of bounded geometry}.
\end{defn}
\begin{rem}
If $M \subset (X,g)$ is an open of bounded geometry, then $(\overline{M},g) \subset (X,g)$ is a submanifold with boundary of bounded geometry.
\end{rem}
\begin{rem}
By Theorem 2.5 of \cite{Schick}, we know that the condition 4. of Definition \ref{open} can be replaced by:
\begin{itemize}
    \item there exist $0 < R_1 \leq inj_{\partial M}$, $0 < R_2 \leq r_C$ and $0 < R_3 < r_i$ and there are some constants $C_k$ for each $k \in \mathbb{N}$ such that whenever we have normal boundary coordinates of width $(r_1, r_2)$ where $r_1 \leq R_1$ and $r_2 \leq R_2$ or Gaussian coordinates of radius $r_3 \leq R_3$ then, in this coordinates
\begin{equation}
    \vert \frac{\partial^\alpha g_{ij}}{\partial x_1^{\alpha_1} \ldots \partial x_m^{\alpha_m}} \vert \leq C_{\alpha} \mbox{     and     }
    \vert \frac{\partial^\alpha g^{ij}}{\partial x_1^{\alpha_1} \ldots \partial x_m^{\alpha_m}} \vert \leq C_{\alpha}
\end{equation}
where $\alpha := \sum_{i = 1}^m \alpha_i$.
\end{itemize}
\end{rem}
\begin{defn}
A \textbf{manifold of bounded geometry with some (possibly) unbounded ends} is an oriented Riemannian manifold $(X,g)$ such that $X = M \cup E_X$ where $M$ is an open of bounded geometry, $E_X$ is an open of $X$ such that $E_X \setminus M \neq \emptyset$ and $\mathcal{N}_X(r_X) \subset M \cap E_X $ for some $r_X \leq r_{C,M}$. Moreover $E_X$ are the \textbf{unbounded ends of $X$}.
\end{defn}
\begin{rem}
Essentially each Riemannian manifold can be seen as a manifold of bounded geometry with some possibly unbounded ends, indeed it is sufficient to choose a point $p$ on $X$ and fix as $M = B_\delta(p)$ where $\delta$ is small enough.
\\This definition will be useful for us when we will talk about uniform map quasi-isometric on the unbounded ends.
\end{rem}
\begin{lem}\label{lemma0}
Let $(X = M \cup E_X, g)$ a manifold of bounded geometry with unbounded ends. Let $p$ be a point in $\mathcal{N}_X(r_X)$. Then since $r_X \leq r_{C,M}$ we can identify $p$ as a point $(x_0,t_0)$ of $\partial M \times (0, r_X)$. Then we have that
\begin{equation}
    d(p, \partial M) = t_0.
\end{equation}
\end{lem}
\begin{proof}
Fix some normal coordinates $(U,x^i)$ centered in $x_0$ in $\partial M$. The normal collar provides some fibered coordinates $(U \times [0, r_c), x^i, t)$ on $\pi^{-1}(U) \subset M$. Notice that, thanks to Proposition 2.8 of \cite{Schick}, we know the Gram matrix of $g$ with respect to $x^i$ and $t$ has the form
\begin{equation}
g(x,t) = \begin{bmatrix} g_{ij}(x,t) & 0 \\
0 & 1 \end{bmatrix}.
\end{equation}
It immediately follows that the curve $\sigma:[0,t_0] \longrightarrow \mathcal{N}(r_X) \subset M$ defined as $\sigma(t) = (0,t)$ is length minimizing.
\end{proof}
\subsection{Uniform maps quasi-isometric on the unbounded ends}
Let us recall the following definitions from Section 1 of \cite{Spes}. Fix two Riemannian manifolds $(X,g)$ and $(Y,h)$.
\begin{defn}
A map $f:(X,g) \longrightarrow (Y,h)$ is \textbf{uniformly continuous} if for each $\epsilon > 0$ there is a $\delta(\epsilon) >0$ such that for each $x_1$, $x_2$ in $X$
\begin{equation}
d_X(x_1, x_2) \leq \delta(\epsilon) \implies d_Y(f(x_1), f(x_2)) \leq \epsilon.
\end{equation}
Moreover $f$ is \textbf{uniformly (metrically) proper}, if for each $R\geq 0$ there is a number $S(R) >0$ such that for each subset $A$ of $(Y, d_Y)$
	\begin{equation}
	diam(A) \leq R \implies diam(f^{-1}(A)) \leq S(R).
	\end{equation}
We say that a map $f:(X, g) \longrightarrow (Y, h)$ is a \textbf{uniform map} if it is uniformly continuous and uniformly proper.
\end{defn}
\begin{defn}
	Two maps $f_0$ and $f_1: (X, d_X) \longrightarrow (Y,d_Y)$ are \textbf{uniformly homotopic} if they are homotopic with a uniformly continuous homotopy $H: (X\times[0,1], g + dt) \longrightarrow (Y, h)$. We will denote it by
	\begin{equation}
	f_1 \sim f_2.
	\end{equation}
\end{defn}
\begin{defn}
	A map $f:(X, g) \longrightarrow (Y, h)$ is a \textbf{uniform homotopy equivalence} if $f$ is uniformly continuous and there is map $s$ such that
	\begin{itemize}
		\item $s$ is a homotopy inverse of $f$,
		\item $s$ is uniformly continuous,
		\item $f\circ s$ is uniformly homotopic to $id_N$ and $s \circ f$ is uniformly homotopic to $id_M$.
	\end{itemize}
\end{defn}
In order to define the class of maps that we will study, we need to introduce the notion of \textit{quasi-isometry}. 
\begin{defn}
Let $(X,g)$ and $(Y,h)$ be two Riemannian manifolds. Then, a \textbf{quasi-isometry} is a local diffeomorphism $f:(X,g) \longrightarrow (Y,h)$ such that $f^*h$ and $g$ are \textbf{quasi-isometric} metric i.e. there is a constant $K \geq 1$ such that
\begin{equation}
    K^{-1}g \leq f^*h \leq Kg.
\end{equation}
\end{defn}
\begin{defn}
Let $(X = M \cup E_X,g)$ and $(Y = N \cup E_Y,h)$ be two manifolds of bounded geometry with unbounded ends. A map $f: (X,g) \longrightarrow (Y,h)$ is \textbf{quasi-isometric on the unbounded ends} if 
\begin{itemize}
    \item $f(E_X) \subseteq E_Y$, $f(M) \subset N$, $f(\partial M) \subseteq \partial N $,
    \item $f_{\vert_{E_X}}: E_X \longrightarrow E_Y$ is a quasi-isometry.
\end{itemize}
A \textbf{uniform homotopy equivalence isometric on the ends} is a uniform homotopy equivalence which is isometric on the unbounded ends. Finally we say that $f$ is a \textbf{uniform homotopy equivalence quasi-isometric on the ends} if $f_{\vert_{E_X}}$ is a quasi-isometry.
\end{defn}
\begin{rem}\label{rem1}
Fix a uniform homotopy equivalence isometric on the ends $f: (X,g) \longrightarrow (Y,h)$. Let $r_0 := \min\{r_X, r_Y\}$. Then $f(\mathcal{N}_X(r_0)) \subseteq \mathcal{N}_Y(r_0)$. Moreover, if we choose some collar coordinates $\{x^i, t\}$ on $\mathcal{N}_X(r_0)$ and $\{y^j, s\}$ on $\mathcal{N}_Y(r_0)$, then $f$ has the form
\begin{equation}
    f(x^i, t) = (g(x^i, t), t).
\end{equation}
This is a consequence of Lemma \ref{lemma0}.
\end{rem}
\begin{lem}\label{ban}
Let $f: X \rightarrow (Y,h)$ be a map between two manifolds. Let $g_1$ and $g_2$ be two quasi-isometric Riemannian metrics on $X$. Then $f: (X,g_1) \longrightarrow (Y,h)$ is a uniform map if and only if $f: (X,g_2) \longrightarrow (Y,h)$ is a uniform map.
\end{lem}
\begin{proof}
Let $\gamma$ be a differentiable curve on $X$. Fix $i= 1, 2$ and denote by $L_i(\gamma)$ the length of $\gamma$ with respect to the metric $g_i$. Since $g_1$ and $g_2$ are quasi-isometric, we obtain that
there is a constant $K$, which does not depends on $\gamma$, such that
\begin{equation}
K^{-1} \cdot L_1(\gamma) \leq L_2(\gamma) \leq K \cdot L_1(\gamma).     
\end{equation}
So, this implies that for each couple of points $a$ and $b$ in $X$ we obtain
\begin{equation}
K^{-1} \cdot d_1(a,b) \leq d_2(a,b) \leq K \cdot d_1(a,b)     
\end{equation}
where $d_i$ is the distance induced by $g_i$. The claim immediately follows.
\end{proof}
\begin{prop}\label{blu}
Let $f:(X = M \cup E_X ,g) \longrightarrow (Y = N \cup E_Y,h)$ be a smooth uniform map which is quasi-isometric on the unbounded ends. Assume, moreover, that, in normal coordinates, $f_{\vert M}$ has uniformly bounded derivatives of each order. Then there is a metric $\tilde{g}$ on $X$ such that 
\begin{enumerate}
    \item $g$ and $\tilde{g}$ are quasi-isometric,
    \item $(\overline{M}, \tilde{g})$ is again a manifold with boundary of bounded geometry, 
    \item there is an open $E_Y$ of $Y$ such that $f: (X = M \cup E_X,\tilde{g}) \longrightarrow (Y = N \cup \tilde{E}_Y,h)$ is a uniform map which is isometric on the unbounded ends.
\end{enumerate}
\end{prop}
\begin{proof}
Let $\phi: \mathbb{R} \longrightarrow \mathbb{R}$ a smooth function such that $\phi \cong 1$ on $[1, +\infty)$, $\phi \cong 0$ on $(-\infty, 0]$ and $\phi(x) \in [0, 1]$ otherwise. Let us define
\begin{equation}
\chi(x) := \begin{cases}
 \phi(\frac{3}{r_0}d_g(\partial M, x) - 1) \mbox{    if    } x \in M \\
 0 \mbox{    otherwise.}
\end{cases}
\end{equation}
Notice that, with respect to the normal coordinates of $(X,g)$, the function $\chi$ has bounded derivatives of each order. This fact is a consequence of Theorem 2.5 of \cite{Schick} and Lemma \ref{lemma0}. Then we define the metric $\tilde{g}$ on $x \in X$ as
\begin{equation}
    \tilde{g}_x := \chi(x)g_x + (1-\chi(x)) f^*h_x.
\end{equation}
Observe that $\tilde{g}$ and $g$ are quasi-isometric. Indeed, since there is a $K$ such that 
\begin{equation}
    K^{-1} \cdot g \leq f^*h \leq K\cdot g,
\end{equation}
then
\begin{equation}
    \frac{1}{1+K} g_x \leq \tilde{g} \leq (1+K) g_{x}.
\end{equation}
Our next step is to prove that $(\overline{M}, \tilde{g})$ is a manifold with boundary of bounded geometry.
\\Let us check the conditions on Definition \ref{open}:
\begin{itemize}
    \item the conditions $1$ and  $2$ are satisfied. Indeed the image of the $\frac{r_0}{3}$-neighborhood of $\partial \overline{M}$ (with respect to $g$) by $f$ contains on a $\delta_0$-neighborhood of $\partial \overline{N}$. This implies that the $\delta_0$-neighborhood of $\partial \overline{M}$ (with respect to $\tilde{g}$) is isometric to the $\delta_0$-neighborhood of $\partial \overline{N}$. Then, since $\overline{N}$ is a manifold with boundary of bounded geometry, the first two conditions of Definition \ref{open} are satisfied.
    \item Condition $3.$ is satisfied. Since the $\delta_0$-neighborhood of $\partial M$ is isometric to the $\delta_0$-neighborhood of $\partial N$, we just have to check that the injectivity radius on $M$ is bounded from below.
    \\Let $p$ be a point on $M$, let $\{e_i\}$ be an orthonormal basis of $T_pX$ and fix some normal coordinates $\{U, x^i\}$ (with respect to $g$) referred to $p$ and $\{e_i\}$. We obtain the coordinates $\{x^i, \mu^j\}$ on $TX$, where $\{\mu^j\}$ are the components of $\frac{\partial}{\partial x^i}$.
    \\Observe that, with respect to these coordinates, the components of the metric $\tilde{g}_{ij}$ are uniformly bounded and the same holds for its derivatives of each order. This because $f$ and $\chi$ have uniformly bounded derivatives of each order with respect to the coordinates $\{x^i\}$ and because $M$ and $N$ have bounded geometry.
    \\Fix $s = \frac{inj_{M}}{K'}$ where $K'>1$ is a constant such that $\tilde{g} \leq K'\cdot g$. Notice that, on the ball\footnote{the radius is given with respect to $g$} $B_{s}(0_p) \subset T_pM$, the exponential map $exp_{g.p}: B_s(0_p) \longrightarrow M$ with respect to $g$ is injective. As a consequence of Picard-Lindel\"of Theorem we also know that the exponential map $exp'_{p}$ with respect to $\tilde{g}$ is well defined on the same ball $B_s(0_p)$.
\\We know that the Jacobian matrix $J_{exp'_{p}}(0_p)$ is the the identity. Moreover, we also know that the derivatives of the entries of $J_{exp'_p}$ are uniformly bounded by a constant $C$: this is a consequence of Lemma 3.4 of \cite{Schick}. So this means that there is a constant $D$ such that for each $v_p$ in $B_s(0_p)$ 
\begin{equation}
    \vert \vert J_{exp'_p}(v_p) - Id \vert \vert \leq D\cdot \vert \vert v_p \vert \vert .
\end{equation}
 Then, if $R = \min\{\frac{1}{2D}, s\}$ on $B_R(0_p)$ the exponential map $exp'_{p}$ is invertible. Then we can conclude by observing that $B_R(0_p)$ contains a ball of radius $\frac{R}{K}$ with respect to the metric $\tilde{g}$. Indeed this fact implies that 
\begin{equation}
    inj_{(M, \tilde{g})}(p) \geq \frac{\min\{\frac{1}{2D}, \frac{inj_{(M,g)}}{K'}\}}{K'}.
\end{equation}
    \item Condition $4.$ is satisfied. In particular the boundedness of the derivatives of the second fundamental form of $\partial M$ follows because $M$ is isometric to $N$ near $\partial M$. On the other hand, the boundedness of the covariant derivatives of the Riemann tensor of $\tilde{g}$ is a consequence of the fact that the norms induced by $g$ and $\tilde{g}$ on each fiber $F_x$ of a tensor multi-product of $TX$ and $T^*X$ are equivalent with some constants which does not depends on $x$, and that $\chi$ and the components of $f^*h$ have uniformly bounded derivatives in normal coordinates with respect to $g$. 
\end{itemize}
Finally, we can conclude the proof by observing that $f: (X,\tilde{g}) \longrightarrow (N,h)$ is a uniform map over $M$ (Lemma \ref{ban}) and it is an isometry on $(X \setminus M) \cup \mathcal{N}(\delta_0)$ where $\mathcal{N}(\delta_0)$ is the $\delta_0$-neighborhood of $\partial M$ in $M$ with respect to $\tilde{g}$. 
\end{proof}
\section{Minimal and maximal domains}
\subsection{The space of squared-integrable differential forms}
Let $(X,g)$ be an oriented Riemannian manifold and let $\Omega^k_{c}(X)$ be the space of complex differential forms with compact support. The Riemannian metric $g$ induces for every $k \in \mathbb{N}$ a scalar product on $\Omega^k_c(X)$ as follows: for each $\alpha$ and $\beta$ in $\Omega^*_{c}(X)$ then
	\begin{equation}
	\langle \alpha, \beta \rangle_{\mathcal{L}^2(X)} := \int\limits_{X} \alpha \wedge \star \bar{\beta},
	\end{equation}
	where $\star$ is the Hodge star operator induced by $g$.
	\\This hermitian product gives a norm on $\Omega^k_{c}(X)$:
	\begin{equation}
	\begin{split}
	\vert \alpha 	\vert^2_{\mathcal{L}^2(X)} :&= \langle\alpha, \alpha\rangle_{\mathcal{L}^2(X)} \\ 
	&= \int\limits_{M}\alpha \wedge \star \bar{\alpha} < +\infty.
	\end{split}
	\end{equation}
	\begin{defn}
		We will denote by $\mathcal{L}^2\Omega^k(X)$ the Hilbert space given by the closure of $\Omega^*_{c}(X)$ with respect to the norm $\vert \cdot \vert_{\mathcal{L}^2\Omega^k(X)}$. Moreover we can also define the Hilbert space $\mathcal{L}^2(X)$ as 
		\begin{equation}
		\mathcal{L}^2(X) :=   \bigoplus\limits_{k \in \mathbb{N}} \mathcal{L}^2\Omega^k(X).
		\end{equation}
		The norm of $\mathcal{L}^2(X)$ it will be denoted by $\vert \cdot \vert_{\mathcal{L}^2}$ or $\vert \cdot \vert_{\mathcal{L}^2(X)}$.
	\end{defn}
	\begin{rem}
		Since $\Omega_c^k(X)$ is dense in $\mathcal{L}^2\Omega^k(X)$, then $\Omega_c^*(X)$ is dense in $\mathcal{L}^2(X)$.
	\end{rem}
\begin{defn}
	Let $(X,g)$ be an oriented Riemannian manifold whose dimension is $m$. The \textbf{chirality operator} is defined as the operator $\tau_X: \mathcal{L}^2(X) \longrightarrow \mathcal{L}^2(X)$ such that, for each $\alpha$ in $\Omega_c^*(X)$, 
\begin{equation}
	\tau_X(\alpha) := i^{\frac{m}{2}} \star \alpha
	\end{equation}
if $m$ is even and
	\begin{equation}
	\tau_X(\alpha) := i^{\frac{m + 1}{2}} \star \alpha
\end{equation}
if $m$ is odd.
\end{defn}
\begin{rem}
The chirality operator is an $\mathcal{L}^2$-bounded operator. In particular, with respect to $ \vert \vert \cdot \vert \vert_{\mathcal{L}^2}$, it is an isometry and an involution and so we obtain that $\tau_X^2 =1$ and $\vert \vert \tau_X \vert \vert_{\mathcal{L}^2} = 1$. Finally, $\tau_X$ is also self-adjoint.
\end{rem}
\begin{defn}
Let $(X,g)$ and $(Y,h)$ be two Riemannian manifold. Fix $A: dom(A) \subset \mathcal{L}^2(X) \longrightarrow \mathcal{L}^2(Y)$ an operator and let $A^*: dom(A^*) \subset \mathcal{L}^2(Y) \longrightarrow \mathcal{L}^2(X)$ be its adjoint operator. Then we denote by $A^\dagger$ the operator defined as
\begin{equation}
    A^\dagger := \tau_X \circ A^* \circ \tau_Y.
\end{equation}
\end{defn}
\begin{rem}
Notice that if $\pi: (X,g) \longrightarrow (Y,h)$ is a submersion, then $(\pi^*)^\dagger = \pi_{star}$ which is the integration along the fibers of $\pi$. If $\omega$ is smooth form on a Riemannian manifold $(X,g)$ and if $e_{\omega}: \mathcal{L}^2(X) \longrightarrow \mathcal{L}^2(X)$ is defined as $e_\omega(\beta) = \beta \wedge \omega$, then $e_{\omega}^\dagger = r_\omega$ where $r_\omega(\beta) := \omega \wedge \beta$.
\end{rem}
\subsection{A regularizing operator}
In this subsection we will talk about $C^k$-forms over a manifold $X$. This are sections of class $C^k$ of the bundle $\Lambda^*(X)$. If $k >1$, then the differential of a $C^k$-form is defined locally in the usual way and it is a $C^{k-1}$-form.
\\Let $(X,g)$ be a Riemannian manifold and fix $\epsilon > 0$. Gol'dshtein and Troyanov in \cite{Gol} studied the de Rham regularizing operator $R_{\epsilon}^X$. In their paper they proved that, given $k>0$ in $\mathbb{N}$, then for each compactly supported $C^k$-form $\omega$ 
\begin{itemize}
    \item $R_{\epsilon}^X\omega$ is in $\Omega^*_c(X)$,
    \item $\lim\limits_{\epsilon \rightarrow + \infty} \vert\vert R_{\epsilon}^X\omega - \omega \vert\vert_{\mathcal{L}^2} = 0$.
    \item $R_{\epsilon}^Xd\omega = d R_{\epsilon}^X \omega$.
\end{itemize}
For the sake of completeness we would like briefly recall the definition and the proofs of the main properties of $R_{\epsilon}^X$.
\\The operator $R_{\epsilon}^X$ is defined as follows: let $n$ be the dimension of $X$ and fix a mollifier $\rho: \mathbb{R}^n \longrightarrow \mathbb{R}$. Then let $h: B_1(0) \subset \mathbb{R}^n \longrightarrow \mathbb{R}^n$ be a radial diffeomorphism such that $h(x) = x $ if $\vert\vert x\vert \vert \leq \frac{1}{3}$ and 
\begin{equation}
    h(x) = \frac{1}{\vert\vert x\vert \vert} exp(\frac{1}{1 - \vert\vert x\vert \vert^2})\cdot x \mbox{ if } \vert\vert x\vert \vert \geq \frac{2}{3}
\end{equation}
Then, we define the submersion $s: \mathbb{R}^n \times \mathbb{R}^n \longrightarrow \mathbb{R}^n$ as
\begin{equation}
    s(x,v) := \begin{cases} h^{-1}(h(x) + v) \mbox{if } \vert\vert x\vert \vert \leq 1 \\
    x \mbox{ otherwise.}
    \end{cases}
\end{equation}
Let $U$ be a bounded convex domain of $\mathbb{R}^n$ which contains the ball $B_1(0)$. Then, for each $\epsilon > 0$, the \textit{local} regularizing operator $R_{\epsilon}: \mathcal{L}^2(U) \longrightarrow \mathcal{L}^2(U)$ is defined as
\begin{equation}
    R_{\epsilon}\omega := \int_{\mathbb{R}^n} s^*\omega \wedge \rho_\epsilon(v) dv^1 \wedge \ldots \wedge dv^n,
\end{equation}
where $\rho_\epsilon(v) := \rho(\frac{v}{\epsilon})$. In order to obtain the global regularizing operator we need to fix a constant $J$ and consider a countable atlas $\{(V_i, \phi_i)\}$ of $M$ such that for each $i$ there are at most $K$ charts $(V_j, \phi_j)$ such that $V_i \cap V_j \neq \emptyset$. Assume that $ B_1(0) \subseteq \phi_i(V_i) \subset \mathbb{R}^n$ for each $i$ and that $\{\phi_i^{-1}(B_1(0))\}$ is again a cover of $M$. Then we define for each $k$ in $\mathbb{N}$,
\begin{equation}
R_{\epsilon, k} := R_{\epsilon, V_0} \circ \ldots \circ  R_{\epsilon, V_k}
\end{equation}
where $R_{\epsilon, V_i} := \phi_i^* \circ R_{\epsilon} \circ [\phi_i^{-1}]^*$ on $V_i$ and it is the identity outside $V_i$. Then we can define the operator $R^X_{\epsilon} := \lim\limits_{k \rightarrow +\infty}R_{\epsilon, k}.$
\bigskip
\\It is a well known fact that if $\omega$ is a $C^k$-form, then $R^X_{\epsilon}\omega$ is a smooth differential form (\cite{Gol}). Essentially this follows because $R_\epsilon$ is a smoothing operator: the components of its kernel $K(x,y) = k(x,y)_J^Idx^J\boxtimes \frac{\partial}{\partial y^I}$ are given by the components of the form $t^*(\rho_\epsilon(v) dv^1 \wedge \ldots \wedge dv^n)$ where $t$ is the inverse map of $S: \mathbb{R}^n \times \mathbb{R}^n \longrightarrow \mathbb{R}^n \times \mathbb{R}^n$ which is defined for each $(x,y)$ as $S(x,y) := (x, s(x,y))$.
\footnote{Actually $S$ is a diffeomorphism only on $\mathbb{R}^n \times B_1(0)$. However the components of $K$ are given by the components of the form $t^*(\rho(v) dv^1 \wedge \ldots \wedge dv^n)$ on $S(\mathbb{R}^n \times B^1(0))$ and zero outside. The components of $K$ are smooth since the support of $\rho(v) dv^1 \wedge \ldots \wedge dv^n$ is contained in $\mathbb{R}^n \times B^1(0)$.}
\\For the same reason is quite easy to show that if $\omega$ is a compactly supported form, then the same happens to $R^X_\epsilon \omega$. In particular, if $\omega = f_I(x)dx^I$ is a differential form whose support is contained on $B_1(0) \subset \mathbb{R}^n$, then $R_\epsilon\omega = f_{I,\epsilon}(x)dx^I$ where $\vert f_I(x) - f_{I,\epsilon}(x) \vert \leq C_{\epsilon, f_I}$ and $C_{\epsilon, f_I}$ goes to zero if $\epsilon$ goes to zero. In other words if $\epsilon = \frac{1}{n}$ then the components $f_{I,\frac{1}{n}}$ uniformly approximate $f_I$.
This, in particular, implies that
\begin{equation}
    \vert \vert R_\epsilon\omega - \omega \vert \vert_{\mathcal{L}^2(B_1(0))} \rightarrow 0
\end{equation}
when $\epsilon \rightarrow 0$. Then for each compactly supported differential form $\alpha$ on a Riemannian manifold $(X,g)$, immediately follows that
\begin{equation}
    \lim \limits_{\epsilon \rightarrow 0} \vert \vert R^X_\epsilon\alpha - \alpha \vert \vert_{\mathcal{L}^2(X)} = 0.
\end{equation}
Finally, $R_\epsilon$ and the exterior derivative operator commute on $C^k_c(\Lambda^*T^*\mathbb{R}^n)$. Indeed $d$ and $s^*$ commute, $d(\alpha) \wedge \phi(v)dv^1\wedge \ldots \wedge dv^n = d(\alpha \wedge \phi(v)dv^1\wedge \ldots \wedge dv^n)$ since $\phi(v)dv^1\wedge \ldots \wedge dv^n$ is a closed form and finally the integral along $\mathbb{R}^n$ and the exterior derivative operator commute on the space of vertically compactly supported smooth forms on $\mathbb{R}^n \times \mathbb{R}^n$ (Proposition 6.14.1. of the book of Bott and Tu \cite{bottu}). As a consequence of this, also $R^X_\epsilon$ commute with the exterior derivative operator on $\Omega^*(X)$.
\subsection{Maximal and minimal domains of the exterior derivative operator}
Let $(X,g)$ be a Riemannian manifold. The exterior derivative operator $d: \Omega^*_c(X) \longrightarrow \Omega^*_c(X)$ can be seen as an unbounded operator with respect to the $\mathcal{L}^2$-norm on $(X,g)$. So we can define some different closures $\overline{d}: dom(\overline{d}) \subset \mathcal{L}^2(X) \longrightarrow \mathcal{L}^2(X)$ of $(d, \Omega^*_c(X))$. In this manuscript we will concentrate on the maximal and the minimal closure of $d$.
\begin{defn}
The \textbf{minimal domain} of the exterior derivative on $(X,g)$ is defined as the subset $dom(d_{min})$ of $ \mathcal{L}^2(X)$ given by the form $\alpha$ such that there is a sequence of $\{\omega_k\} \subset \Omega^*_c(X)$ such that $\lim\limits_{k \rightarrow +\infty}\vert \vert \omega_k -\alpha \vert \vert_{\mathcal{L}^2(X)} = 0$ and the sequence of $\{d\omega_k\}$ converges in $\mathcal{L}^2(X)$ with respect to the $\mathcal{L}^2$-norm.
\\Then we can define $d_{min}(\alpha) := \lim\limits_{k \rightarrow +\infty} d\omega_k$.
\end{defn}
\begin{defn}
The \textbf{maximal domain} of the exterior derivative on $(X,g)$ is defined as the subset $dom(d_{max})$ of $ \mathcal{L}^2(X)$ given by the $k$-forms $\alpha$ such that there a $\mathcal{L}^2$-form $\eta_\alpha$ such that for each $\beta$ in $\Omega_c^*(X)$
\begin{equation}
    \int_X \alpha \wedge d\beta = (-1)^{k +1} \int_X \eta_\alpha \wedge \beta.
\end{equation}
Then we can define $d_{max}(\alpha) := \eta_\alpha$.
\end{defn}
A well-known result about the minimal and the maximal closure of $d$, is the following. Let $(\overline{d}, dom(\overline{d}))$ be a closed extension of $d$ on $\mathcal{L}^2(X)$, then
\begin{equation}
(dom(d_{min}), d_{min}) \subseteq (\overline{d}, dom(\overline{d})) \subseteq (dom(d_{max}), d_{max}).
\end{equation}
Moreover Gaffney in \cite{Gaf} proved that if $(X,g)$ is a complete Riemannian manifold, then
$dom(d_{min}) = dom(d_{max})$ and so there is just one closed extension of $d$.
\begin{lem}
Let $(X,g)$ be a Riemannian manifold and let $\alpha \in dom(d_{max})$. Then, for each $k \geq 1$ and for each $\beta$ in  $C^k_c(\Lambda^*T^*M)$, we have
\begin{equation}
    \int_{M} \alpha \wedge d\beta = (-1)^k \int_M d_{max}\alpha \wedge \beta.
\end{equation}
\end{lem}
\begin{proof}
Let us denote by $R_\epsilon^X$ the de Rham regularizing operator. We can assume, without loss of generality, that there is a $j$ such that $supp(\beta)$ is contained in $\phi_j^{-1}(B_1(0)) \subseteq V_j$. Here $\{(V_i, \phi_i)\}$ is the atlas used to define $R_\epsilon^X$. Observe that $\alpha \wedge R^X_{\frac{1}{n}}\beta$ converges point-wise to $\alpha \wedge \beta$. 
\\Notice that on $V_i$, the form $\alpha$ has the form $\alpha(x) = \alpha_I(x)dx^I$ such that
\begin{equation}
\int_{V_{i}} g^{IJ}(x) \alpha_I(x) \alpha_J(x) dx^1 \wedge \ldots \wedge dx^n
\end{equation}
is bounded. Here the capital letter $I$ is a multi-index defined unless double switching. Moreover, we also denote by $N_I$ the multi-index such that $(I,N-I) = (1, \ldots, n)$. Again $N_I$ is defined unless to double switching. So the integral of $\alpha \wedge d\beta$ has the form $\int_{\phi_i^{-1}(B_1(0))} \alpha_I \cdot  (d\beta)_{N-I} dx^1 \wedge \ldots \wedge dx^n$. Notice that $\vert \alpha_I \cdot  (d\beta)_{N-I} \vert \leq C_\beta \cdot \vert \alpha_I \vert$. This means that if $\vert \alpha_I \vert $ is in $L^1(\phi_j^{-1}(B_1(0)))$ then we could apply the Dominated Convergence Theorem to the sequence $\alpha_I \cdot  \lim\limits_{k \rightarrow +\infty} (dR^X_{\frac{1}{k}}\beta)_{N-I}$.
\\Observe that $\vert \alpha_I \vert$ is in $L^2(\phi_j^{-1}(B_1(0)))$. Indeed 
\begin{equation}
\begin{split}
    \int_{\phi_j^{-1}(B_1(0))} \vert \alpha_I \vert^2 dx^1 \wedge \ldots \wedge dx^n &\leq \frac{1}{m_j} \int_{\phi_j^{-1}(B_1(0))} m_j \vert \alpha_I \vert^2 dx^1 \wedge \ldots \wedge dx^n \\
    &\leq \frac{1}{m_j} \int_{\phi_j^{-1}(B_1(0))} g^{IJ} \alpha_I \alpha_J dx^1 \wedge \ldots \wedge dx^n \\
    &\leq C_j
\end{split}
\end{equation}
where $m_i$ is the infimum of the eigenvalues of $g$ on the closure of $V_i$. So, since $\phi_j^{-1}(B_1(0))$ is compact, then $\vert \alpha_I \vert$ is also an $L^1$ function. This means that, thanks to the Dominated Convergence Theorem,
\begin{equation}
    \begin{split}
        \int_{V_i} \alpha \wedge d\beta &= \int_{V_i}\alpha_I \cdot  (d\beta)_{N-I} dx^1 \wedge \ldots \wedge dx^n\\ &= \int_{V_i}\alpha_I \cdot  \lim\limits_{k \rightarrow +\infty} (dR^X_{\frac{1}{k}}\beta)_{N-I} dx^1 \wedge \ldots \wedge dx^n\\
        &= \int_{V_i}\lim\limits_{k \rightarrow +\infty} \alpha_I \cdot (dR^X_{\frac{1}{k}}\beta)_{N-I} dx^1 \wedge \ldots \wedge dx^n)\\
        &= \lim\limits_{k \rightarrow +\infty} \int_{V_i} \alpha \wedge dR^x_{\frac{1}{k}}\beta.
    \end{split}
\end{equation}
Observe that $dR^X_{\frac{1}{k}}\beta$ is a compactly supported smooth form: this means that
\begin{equation}
    \int_{V_i} \alpha \wedge d\beta = \lim\limits_{k \rightarrow +\infty} \int_{V_i} d\alpha \wedge R^X_{\frac{1}{k}}\beta.
\end{equation}
So, by applying again the Dominated Convergence Theorem, we conclude.
\end{proof}
\begin{cor}\label{corollary1}
Let $(X,g)$ and $(Y,h)$ be two Riemannian manifolds. Fix an $\mathcal{L}^2$-bounded operator $A: \Omega^*_c(Y) \longrightarrow C^k_c(\Lambda^*(T^*X))$ for some $k > 1$. Suppose that $A^\dagger: \Omega^*_c(X) \longrightarrow C^k_c(\Lambda^*(T^*Y))$. Finally assume that, for each smooth form $\omega$
\begin{equation}
dA\omega \pm Ad \omega = B\omega
\end{equation}
where $B\omega$ is a $\mathcal{L}^2$-bounded operator. Then $A(dom(d_{min})) \subseteq dom(d_{min})$ and $A(dom(d_{max})) \subseteq dom(d_{max})$.
\end{cor}
\begin{proof}
Let $\alpha$ be a $k$-form of $dom(d_{maxY})$: there is an $\mathcal{L}^2$-form $d\alpha$ such that, for each $\beta$ in $\Omega^*_c(Y)$,
\begin{equation}
    \int_{Y} d\alpha \wedge \beta = (-1)^k \int_Y \alpha \wedge d\beta.
\end{equation}
Let us concentrate on $A \alpha$. Let $\gamma$ be a form in $\Omega^*_c(X)$. We obtain that
\begin{equation}
    \begin{split}
    \int_X A \alpha \wedge d\gamma &= \langle A \alpha, \tau_X d\gamma \rangle_X\\
    &= \langle \alpha, A^* \tau_X d\gamma \rangle_Y\\
    &= \langle \alpha, \tau_Y A^\dagger \tau_X \tau_X d\gamma \rangle_Y\\
    &= \langle \alpha, \tau_Y A^\dagger d\gamma \rangle_Y\\
    &= (-1)^k \langle \alpha, \tau_Y A^\dagger d^\dagger \gamma \rangle_Y\\
    &=(-1)^k \langle \alpha, \tau_Y  (d^\dagger \mp A^\dagger + B^\dagger) \gamma \rangle_Y\\
    &=\langle \alpha, \tau_Y (d \mp A^\dagger + B^\dagger) \gamma \rangle_Y\\
    &= \langle d_{max}\alpha, \tau_Y \mp A^\dagger \gamma \rangle + \langle B \alpha, \tau_Y \gamma \rangle\\
    &= \langle (\mp Ad_{max} + B)\alpha, \gamma \rangle.
    \end{split}
\end{equation}
So, this means that $A(dom(d_{max})) \subseteq dom(d_{max})$ and $d_{max}A = \mp Ad_{max} + B$.
\bigskip
\\Let $\beta$ be an element in $dom(d_{minY})$. Then there is a sequence $\{\beta_n\}$ such that $\beta = \lim\limits_{n\rightarrow +\infty} \beta_n$ and $d\beta = \lim\limits_{n\rightarrow +\infty} d\beta_n$ with respect to the $\mathcal{L}^2$-norm. Since $A$ is a $\mathcal{L}^2$-bounded operator, we obtain that
\begin{equation}
    A \beta = \lim\limits_{n \rightarrow + \infty} A \beta_n.
\end{equation}
Fix $n \in \mathbb{N}$ and let $K(n)$ be a number such that
\begin{equation}
\vert \vert A \beta - A \beta_{K(n)} \vert \vert_{\mathcal{L}^2} \leq \frac{1}{2n}
\end{equation}
and 
\begin{equation}
\vert \vert A d \beta - A d\beta_{K(n)}\vert \vert_{\mathcal{L}^2} \leq \frac{1}{2n}.
\end{equation}
Then let $\epsilon(n)$ be small enough to
\begin{equation}
\vert \vert A \beta_{K(n)} - R_{\epsilon(n)}^XA \beta_{K(n)} \vert \vert_{\mathcal{L}^2} \leq \frac{1}{2n}
\end{equation}
and
\begin{equation}
\vert \vert A d\beta_{K(n)} - R_{\epsilon(n)}^XA d\beta_{K(n)} \vert \vert_{\mathcal{L}^2} \leq \frac{1}{2n}.
\end{equation}
So let $\{\gamma_n\}$ be the sequence defined as $R_{\epsilon(n)}^XA \beta_{K(n)}$ for each $n$. This is a sequence of compactly supported smooth forms which converges, with respect to the $\mathcal{L}^2$-norm, to $A \beta$. We just have to show that $dR_{\epsilon(n)}^X A \beta_{K(n)}$ converges to an element of $\mathcal{L}^2(X)$.
\\Observe that
\begin{equation}
\begin{split}
d R_{\epsilon(n)}^X A \beta_{K(n)} &= R_{\epsilon(n)}^X d_{max,X} A \beta_{K(n)}\\
&= R_{\epsilon(n)}^X (\mp A d_{max,X} + B)\beta_{K(n)}\\
&= R_{\epsilon(n)}^X (\mp A d + B) \beta_{K(n)}.
\end{split}
\end{equation}
This means that 
\begin{equation}
    \begin{split}
    \vert \vert A d \beta - R_{\epsilon(n)}^X (\mp A d + B) \beta_{K(n)} \vert \vert_{\mathcal{L}^2} &\leq 
    \vert \vert A d \beta - A d\beta_{K(n)}\vert \vert_{\mathcal{L}^2}\\ &+ \vert \vert A d\beta_{K(n)} - R_{\epsilon(n)}^X(\mp A d + B)\beta_{K(n)} \vert \vert_{\mathcal{L}^2}\\
    &\leq \frac{1}{2n} + \frac{1}{2n} = \frac{1}{n}.
    \end{split}
\end{equation}
So we just proved that $A(dom(d_{minY})) \subseteq dom(d_{minX})$ and $d_{min} A = \mp A d_{min} + B.$
\end{proof}
\subsection{Reduced and unreduced cohomologies}
Let $(X,g)$ be a Riemannian manifold. It is a well-known fact that 
\begin{equation}
d_{min \backslash max}(dom(d_{min \backslash max}))\subseteq dom(d_{min \backslash max})
\end{equation}
and that $d_{min \backslash max}^2 := d_{min \backslash max} \circ d_{min \backslash max} = 0$. So, we can define the cohomology groups of $L^2$-cohomology as follows.
\begin{defn}
The \textbf{$k$-th group of minimal $L^2$-cohomology} is the group defined as 
\begin{equation}
H^k_{2, min}(X,g) := \frac{ker(d_{min}^k)}{im(d^{k-1}_{min})}.
\end{equation}
Moreover the \textbf{$k$-th group of reduced minimal $L^2$-cohomology} is given by
\begin{equation}
\overline{H}^k_{2, min}(X,g) := \frac{ker(d_{min}^k)}{\overline{im(d^{k-1}_{min})}}.
\end{equation}
where $d^k_{min}$ is the operator $d^k_{min}$ defined on $dom(d_{min}) \cap \mathcal{L}^2(\Omega^k(X))$.
\end{defn}
\begin{defn}
The \textbf{$k$-th group of maximal $L^2$-cohomology} is the group defined as 
\begin{equation}
H^k_{2, max}(X,g) := \frac{ker(d_{max}^k)}{im(d^{k-1}_{max})}.
\end{equation}
Moreover the \textbf{$k$-th group of reduced maximal $L^2$-cohomology} is given by
\begin{equation}
\overline{H}^k_{2, max}(X,g) := \frac{ker(d_{max}^k)}{\overline{im(d^{k-1}_{max})}}.
\end{equation}
where $d^k_{max}$ is the operator $d^k_{max}$ defined on $dom(d_{max}) \cap \mathcal{L}^2(\Omega^k(X))$.
\end{defn}
In general, for each $k \in \mathbb{N}$, the groups $H^k_{2, min}(X,g)$, $\overline{H}^k_{2, min}(X,g)$, $H^k_{2, max}(X,g)$ and $\overline{H}^k_{2, max}(X,g)$ can be different. 
\\On the other hand, we know that if $(X,g)$ is a complete Riemannian manifold then $d_{min} = d_{max}$. In this case, we obtain
\begin{equation}
    H^k_{2, min}(X,g) = H^k_{2, max}(X,g) \mbox{   and    } \overline{H}^k_{2, min}(X,g) = \overline{H}^k_{2, max}(X,g).
\end{equation}
\begin{rem}
If two Riemannian manifolds $(X,g)$ and $(Y,h)$ are quasi-isometric manifolds, then their $L^2$-cohomology groups (minimal or maximal, reduced or unreduced) are isomorphic.
\end{rem}
\section{R.-N.-Lipschitz maps and vector bundles}
\subsection{Radon-Nikodym-Lipschitz maps}
In this Subsection we will recall some notions introduced in subsections 2.3. and 2.4 of \cite{Spes}. This notions will be useful in order to understand when the pullback along a map between Riemannian manifolds induces a $\mathcal{L}^2$-bounded operator.
\\Let $(M, \nu)$ and $(N, \mu)$ be two measured spaces and let $f: (M, \nu) \longrightarrow (N, \mu)$ be a function such that the pushforward measure $f_\star(\nu)$ is absolutely continuous with respect to $\mu$.
\begin{defn}
	Let $(N, \mu)$ be $\sigma$-finite, then the \textbf{Fiber Volume of } $f$ is the Radon-Nikodym derivative
	\begin{equation}
	Vol_{f, \nu, \mu} := \frac{\partial f_\star \nu}{\partial \mu}.
	\end{equation}
\end{defn}
Let $(M, d_M, \mu_M)$ and $(N, d_N, \mu_N)$ be two measured and metric spaces. 
\begin{defn}
	A map $f: (M, d_M, \mu_M) \longrightarrow (N, d_N, \mu_N)$ is \textbf{Radon-Nikodym-Lipschitz} or \textbf{R.-N.-Lipschitz} if
	\begin{itemize}
		\item $f$ is Lipschitz
		\item $f$ has a well-defined and bounded Fiber Volume.
	\end{itemize}
\end{defn}
\begin{rem}
As showed in Remark 12 of \cite{Spes}, an equivalent definition is the following. A map $f: (M, d_M, \nu_M) \longrightarrow (N, d_N, \mu_N)$ is a \textbf{R.-N.-Lipschitz map} if it is Lipschitz and there is a constant $C$ such that, for each measurable set $A \subseteq N$, 
	\begin{equation}\label{follow}
	\mu_M(f^{-1}(A)) \leq C \mu_N(A).
	\end{equation}
So this implies that a composition of R.-N.-Lipschitz maps is a R.-N.-Lipschitz map.
\end{rem}
The main property of R.-N.-Lipschitz map is proved in Proposition 2.4 of \cite{Spes}.
\begin{prop}
	Let $(M,g)$ and $(N,h)$ be Riemannian manifolds. Let $f:(M,g) \longrightarrow (N,h)$ be a R.-N.-Lipschitz map. Then $f^*$, which is the pullback along $f$, is an $\mathcal{L}^2$-bounded operator.
\end{prop}
In general a Lipschitz map is not R.-N.-Lipschitz: for example if $(M,g)$ and $(N,h)$ are two Riemannian manifolds such that $dim(M) < dim(N)$ then there is no Lipschitz embedding $i:(M,g) \longrightarrow (N,h)$ which is R.-N.-Lipschitz.
\\However we are interested in the Fiber Volume of a submersion. Indeed if $f: (M,g) \longrightarrow (N,h)$ is a Lipschitz submersion between oriented Riemannian manifold, then we know how to compute its Fiber Volume. In order to explain how compute it, we need the notion of quotient between two differential forms.
	\begin{defn}
		Let us consider a differentiable manifold $X$. Given two differential forms $\alpha \in \Omega^k(X)$, $\beta \in \Omega^n(X)$ we define \textbf{a quotient between $\alpha$ and $\beta$}, denoted by $\frac{\alpha}{\beta}$, as a (possibly not continuous) section of $\Lambda^{k-n}(X)$ such that for all $p$ in $M$
		\begin{equation}
		\alpha(p) =  \beta(p) \wedge \frac{\alpha}{\beta}(p).\label{quotient}.
		\end{equation}
	\end{defn}
In general, given two smooth differential forms, we do not know if there is a quotient between them. Moreover, if a quotient between $\alpha$ and $\beta$ exists, it may be not unique. However, given a submersion between Riemannian manifolds $\pi: (X,g) \longrightarrow (Y,h)$, then it is possible to define a locally smooth quotient between $Vol_X$ and $\pi^*Vol_Y$ and the pullback of this quotient on a fiber of $\pi$ is a smooth form which does not depends on the choice of the quotient (pag. 15 of \cite{Spes}). So this means that if we denote by $i_q: F_q \longrightarrow X$ the embedding of the fiber of $q$ in $X$, then 
\begin{equation}
\int_{F_q} i^*_q(\frac{Vol_X}{\pi^*Vol_Y})
\end{equation}
is a well-defined real number for each $q$ in $N$.
\begin{prop}\label{cosa}
	Let $(X,g)$ and $(Y,h)$ two oriented, Riemannian manifolds. Let $\pi: (X,g) \longrightarrow (Y,h)$ be a submersion. Then, if $q$ is on $im(\pi)$, then
	\begin{equation}
	Vol_{\pi,\mu_M,\mu_N}(q) = \int_F \frac{Vol_X}{\pi^*Vol_Y}(q)
	\end{equation}
	and $Vol_{\pi,\mu_M,\mu_N}(q) = 0$ otherwise.
\end{prop}
\begin{rem}\label{oss}
	If the submersion $f: X \rightarrow Y$ is a diffeomorphism between oriented manifold which preserves the orientations, then the integration along the fibers of $f$ is the pullback $(f^{-1})^*$. This means that the Fiber Volume of $f$ is given by $\vert(f^{-1})^*\frac{Vol_X}{f^*(Vol_Y)}\vert$.
\end{rem}
We also recall Proposition 2.6 of \cite{Spes}: this Proposition allow to compute the Fiber Volume of the composition of two submersions.
\begin{prop}\label{compo}
	Let $f:(M,g) \longrightarrow (N,h)$ and $g:(N,h) \longrightarrow (W, l)$ be two submersions between oriented Riemannian manifolds. Then
	\begin{equation}
	Vol_{g \circ f, \mu_M, \mu_W}(q) = \int_{g^{-1}(q)}(\int_{f^{-1}g^{-1}(q)} \frac{Vol_M}{f^*(Vol_N)})\frac{Vol_N}{g^*Vol_N}
	\end{equation}
\end{prop}
Finally we conclude this subsection by showing the relation between R.-N.-Lipschitz maps and the quasi-isometries.
\begin{prop}
Let $f:(X,g) \longrightarrow (Y,h)$ be a quasi-isometry. Then $f$ is a R.N.-Lipschitz map.
\end{prop}
\begin{proof}
Notice that $f$ is Lipschitz: it directly follows from the definition of quasi-isometry. So we just have to show that $f$ has a bounded Fiber Volume. This can be proved by showing that $id:(X,g) \longrightarrow (X, f^*h)$ has bounded Fiber Volume. Let $x$ be a point of $X$. By a simultaneous diagonalization argument we can choose a basis $\{e_1, \ldots, e_{n}\}$ of $T_xX$ such that the Gram matrices of $f^*h$ and $g$ are both diagonals. Observe that for each $i = 1, \ldots, n$ we obtain that $f^*h(e_i, e_i) \leq K g(e_i, e_i)$. So this means that
\begin{equation}
     det(f^*h_{jl}(x)) = \prod_{i=1}^n f^*h(e_i, e_i) \geq K^{-n} \prod_{i=1}^n g(e_i, e_i) = det(g_{jl}(x)).
\end{equation}
Notice that the Fiber Volume of $id$ on a point $y = f(x)$ is given by
\begin{equation}
\sqrt{\frac{det(g_{jl}(x))}{det(f^*h_{jl}(x))}} \leq K^\frac{n}{2}.
\end{equation}
This conclude the proof. 
\end{proof}
\subsection{Generalized Sasaki metrics}
The definition of \textit{Sasaki metric} on a vector bundle given in this paper is a generalization of the metric defined by Sasaki in \cite{SSK} for the tangent bundle of a Riemannian manifolds. The definition of this generalized Sasaki metric can be found in page 2 of the paper of Boucetta and Essoufi \cite{Boucetta}.
	\\Let us consider a Riemannian manifold $(N,h)$ of dimension $n$, $\pi_E : E \longrightarrow N$ a vector bundle of rank $m$ endowed with a bundle metric $H_E \in \Gamma(E^*\otimes E^*)$ and a linear connection $\nabla_E$ which preserves $H_E$. Fix $\{s_\alpha \}$ a local frame of $E$: if $\{x^i\}$ is a system of local coordinates over $U \subseteq N$, then we can define the system of coordinates $\{x^i, \mu^\alpha \}$ on $\pi_E^{-1}(U)$, where the $\mu^\alpha $ are the components with respect to $\{s_\alpha \}$.
	\\Let us denote by $K$ the map $K: TE \longrightarrow E$ defined as
	\begin{equation}
	K(b^i\frac{\partial}{\partial x^i }\vert_{(x_0, \mu_0)} + z^\alpha \frac{\partial}{\partial \mu^\alpha }\vert_{(x_0, \mu_0)}) := (z^\alpha + b^i \mu^j \Gamma_{ij}^\alpha (x_0))s_\alpha (x_0),
	\end{equation}
	where the $\Gamma_{ij}^l$ are the Christoffel symbols of $\nabla_E$. The Christoffel symbols $\Gamma^\gamma_{\eta j}(x)$ are defined by the formula
		\begin{equation}
		 \nabla^E_{\frac{\partial}{\partial x^j }} s_{\eta}(x) := \Gamma^\gamma_{\eta j}(x) s_\gamma(x).
		\end{equation}
	\begin{defn}
		The \textbf{Sasaki metric} on $E$ is the Riemannian metric $h^E$ defined for all $A,B$ in $T_{(p, v_p)}E$ as
		\begin{equation}
		h^E(A,B) := h(d\pi_{E,v_p}(A), d\pi_{E,v_p}(B)) + H_E(K(A), K(B)).
		\end{equation}
	\end{defn}
	\begin{rem}\label{riemsub}
		Let us consider the system of coordinates $\{x^i\}$ on $N$ and $\{x^i, \mu^j\}$ on $E$. The components of $h^E$ are given by
		\begin{equation}\label{metri}
		\begin{cases}
		h^E_{ij}(x,\mu) = h_{ij}(x) + H_{\alpha\gamma}(x)\Gamma^\alpha_{\beta i}(x)\Gamma^\gamma_{\eta j}(x)\mu^\beta\mu^\eta \\
		h^E_{i\sigma}(x, \mu) = H_{\sigma \alpha}(x)\Gamma^\alpha_{\beta i}(x)\mu^\beta\\
		h^E_{\sigma\tau}(x,\mu) = H_{\sigma, \tau}(x), 
		\end{cases}
		\end{equation}
		where $i,j = 1,\ldots, n$ and $\sigma, \tau = n+1,\ldots,n+m$.
		Consider a point $x_0 = (x^1_0, \ldots, x^n_0)$ in $N$. If all the Christoffel symbols of $\nabla_E$ in $x_0$ are zero, then, in local coordinates, the matrix of $h^E$ in a point $(x_0, \mu)$ is
		\begin{equation}\label{metrinulla}
		\begin{bmatrix}
		h_{i,j}(x_0) && 0 \\
		0 && H_{\sigma, \tau}(x_0)
		\end{bmatrix}.
		\end{equation}
Moreover, with respect to the coordinates $(x^i, \mu^\sigma)$, the matrix $h_E := (h^E)^{-1}$ is given by
\begin{equation}\label{invmetri}
    \begin{cases}
        h^{ij}_E(x,\mu) =  h^{ij}(x) \\
        h^{i\sigma}_E = -\Gamma^\sigma_{\beta j}(x) h^{ij}(x)\mu^\beta \\
        h_E^{\sigma \tau} = H^{\sigma \tau}(x) +  h^{ij}(x) \Gamma^\sigma_{\beta i}(x)\Gamma^\tau_{\eta j}(x)\mu^\beta\mu^\eta
    \end{cases}
\end{equation}
where $H^{\sigma \tau}$ and $ h^{ij}(x)$ are the components of the inverse matrices of $h_{ij}$ and $H_{\sigma, \tau}$.
	\end{rem}
	\begin{exem}\label{twos}
		Let $(M,g)$ a Riemannian manifold. Consider as $E$ the tangent bundle $TM$ and as $h_E$ the metric $g$ itself. Choose the connection $\nabla_E$ as the Levi-Civita connection $\nabla_g^{LC}$. We denote by $g_S$ the Sasaki metric induced by $g$ and $\nabla^{LC}_g$.
	\end{exem}
	\begin{exem}\label{threes}
	Consider a smooth map $f:(M,g) \longrightarrow (N,h)$. Let $\pi: f^*TN \longrightarrow M$ be the pullback bundle. Then the Riemannian metric $h$ can be seen as a bundle metric on $TN$ and so we obtain a bundle metric $f^*h$ on $f^*TN$. Fix the connection $f^*\nabla^{LC}_h$ on $f^*TN$ which is the pullback of the Levi-Civita connection on $(N,g)$. Let us denote by $g_{S,f}$ the Sasaki metric induced by $f^*\nabla^{LC}_h$, $f^*h$ and $g$.
	\end{exem}
	\begin{rem}
	 Then the Christoffel symbols of the pullback connection $f^*\nabla^E$ with respect to the pullback frame $\{f^* e_i\}$ and to the coordinates $\{x^i\}$ are given by 
	\begin{equation}
	 \tilde{\Gamma}^\alpha_{\beta, i}:= \frac{\partial f^l}{\partial x^i}f^*(\Gamma^\alpha_{\beta, l}).
	\end{equation}
	A first consequence of this fact is that the map $ID: (id^*TN, g_{S, id}) \longrightarrow (TN, g_S)$ which sends $(p, w_{p})$ to $w_p$ is an isometry. So, from now on, we will identify $(id^*TN, g_{S, id})$ and $(TN, g_S)$.
	\end{rem}
	\begin{rem}\label{respect}
	Let $f: (M,g) \longrightarrow (N,h)$ be a smooth map. Let us fix a chart $\{U, x^i\}$ on $M$ and a chart $\{V, y^j\}$ on $N$ such that $f(U) \subseteq V$. Fix a bundle $E$ on $N$ and let $\nabla^E$ be a connection. Let us denote by $\Gamma^\alpha_{\beta l}$ the Christoffel symbols of $\nabla^E$ with respect to a frame $\{e_i\}$ and the coordinates $\{y^j\}$. Fix $\delta > 0$ and let $T^\delta N$ be the subset of $TN$ given by the vector whose norm is less or equal to $\delta$. If $f: (M,g) \longrightarrow (N,h)$ is a smooth Lipschitz map, then also the induced bundle map 
	\begin{equation}
	    \begin{split}
	        F:(f^*T^\delta N, g_{S,f}) &\longrightarrow (T^\delta N, g_S) \\
	        (p, w_{f(p)}) &\longrightarrow w_{f(p)}
	    \end{split}
	\end{equation}
	is a smooth Lipschitz map.
	\end{rem}
	\begin{prop}\label{geodesis}
	 Consider a vector bundle $(E, \pi, M)$ over a Riemannian manifold $(M,g)$. Fix on $E$ a bundle metric $h$, a connection $\nabla^E$ and let us denote by $h_E$ the Sasaki metric induced by $g$, $h$ and $\nabla^E$. Let us suppose that for each point $p$ on $M$ there is a system of normal coordinates $\{x^i\}$ around $p$ and a local frame $\{s_\sigma\}$ such that the Christoffel symbols of $\nabla^E$ vanishes at $x= 0$. Moreover let us suppose that $\frac{\partial h_{\sigma \tau}}{\partial x^k}(0) = 0$, where $h_{\sigma \tau}$ are the components of the Gram matrix of $h$ with respect to the coordinates $\{x^i\}$. Then the fibers of $\pi$ are totally geodesic submanifolds (in our case this means that the straight lines on the fibers are geodesics) and $\pi$ is a Riemannian submersion.
	\end{prop}
		\begin{rem}
	The vector bundles of the Examples \ref{twos} and \ref{threes} satisfy the assumptions of Proposition \ref{geodesis}.
	\end{rem}
	\begin{cor}
	Consider $0_E$ the null section of a vector bundle $\pi: E \longrightarrow M$. Under the same assumptions of the Proposition \ref{geodesis}, the disk bundle $E^\delta := \{v_p \in E \vert \sqrt{h(v_p,v_p)}  \leq \delta \}$ coincides with $B_{\delta}(0_E) := \{v_p \in E \vert d_{h^E}(v_p, 0_E) \leq \delta \}$ for each $\delta > 0$.
	\end{cor}
	\begin{prop}\label{star}
	Consider a Riemannian manifold $(M,g)$ and let $\pi: E \longrightarrow M$ be a vector bundle. Fix on $E$ a metric bundle $h_E$ and a connection $\nabla_E$. Let $h^E$ be the Sasaki metric on $E$ defined by using $g$, $h_E$ and $\nabla_E$. Fix a $\delta > 0$. Then, under the assumptions of Proposition \ref{geodesis} $\pi: E^\delta \longrightarrow M$ is a R.-N.-Lipschitz map. In particular the Fiber Volume on a point $q$ is the Volume of an Euclidean ball of radius $\delta$. Finally the integration along the fibers $\pi_\star: \Omega_{c}^*(E^\delta) \longrightarrow \Omega^*(M)$ is an $\mathcal{L}^2$-bounded operator.
	\end{prop}
\subsection{Mathai-Quillen-Thom forms}\label{form}
Let us start by introducing the notion of Thom form.
	 \begin{defn}
	 Let $\pi: E \longrightarrow M$ be a vector bundle. A smooth form $\omega$ in $\Omega_{cv}^*(E)$ is a \textbf{Thom form} if it is closed and its integral along the fibers of $\pi$ is equal to the constant function $1$.
	 \end{defn}
	 Given a Thom form $\omega$ of $f^*TN$ such that $supp(\omega)$ is contained in a $\delta_0 < \delta$ neighborhood of the null section, let us define an operator $e_\omega: \Omega^*(f^*(T^\delta N)) \longrightarrow \Omega^*(f^*(T^\delta N))$, for every smooth form $\alpha$ as
	  \begin{equation}
	  	e_\omega(\alpha) := \alpha \wedge \omega.
	  \end{equation}
	  Here $f^*(T^\delta N)$ is the $\delta$-neighborhood of the $0$-section of $f^*TN$.
	 \\In this paper, we use the Thom forms introduced by Mathai and Quillen in \cite{mathai}. In their work, indeed, they defined a Thom form for a vector bundle $\pi: E \longrightarrow (M,g)$ induced by the Riemannian metric $g$, a metric bundle $h_E$ and a connection $\nabla_E$. In particular, for each $\delta > 0$ they found a Thom form whose support is contained on
	 \begin{equation}
	     B_\delta(0_E) := \{w \in E \vert h_E(w,w) \leq \delta^2\}.
	 \end{equation}
	 So, we will consider, on the tangent bundle of a Riemannian manifold $(Y,h)$ the Mathai-Quillen-Thom $\omega_Y$ form induced by $h$ and $\nabla^{LC}_h$. On the other hand, given a map $f:X \longrightarrow Y$, we fix on $f^*TY$ the pullback $F^*\omega_Y$ where $F:f^*(TY) \longrightarrow TY$ is the bundle map induced by $f$.
	 \\As shown in subsection 4.2. of \cite{Spes}, if $h$ has the same bounds of a manifolds of bounded geometry and if $f$ is a smooth Lipschitz map, then there is a constant $C$ such that for each $w$ in $f^*TY$ the norm $\vert f^*\omega_Y\vert_q \leq C$. In this case the operator $\Omega^*_c(TY) \longrightarrow \Omega^*_c(TY)$ is an $\mathcal{L}^2$-bounded operator (Propositions 4.3 and 4.4 of \cite{Spes}).
\section{Stability of cohomologies}
\subsection{A submersion related to a uniform map} 
Let $(X,g)$ and $(Y,h)$ two manifolds of bounded geometry with unbounded ends. Fix $f:(X,g) \longrightarrow (Y,h)$ which is a uniform map isometric on the ends.
\begin{prop}\label{appr}
Assume the existence of a close subset $C$ on $X$ and a number $R$ such that $f_{X \setminus C}$, in local normal coordinates defined on a ball of radius at most $r_i$, the map $f$ has uniformly bounded derivatives of each order. Then, for each $\epsilon >0$ there is a smooth map $f_{\epsilon}: (X,g) \longrightarrow (Y,h)$ such that $f_{\epsilon} \sim f$ and $f_{\epsilon} = f$ on $X \setminus B_{\epsilon}(C)$.
\end{prop}
\begin{proof}
This is essentially the same proof given in Proposition 1.7 and in Corollary 1.8 of \cite{Spes}. Observe, indeed, that since $f$ is a quasi-isometry on the ends, then $C \subset M$. Then, since the proof of Proposition 1.7 works locally, we can assume that $C$ is a subset of a manifold of bounded geometry and we can consider $f$ as a $C^0_{b}$-map.
\end{proof}
Fix $f:(X = M \cup E_X ,g) \longrightarrow (Y = N \cup E_Y,h)$ be a map which is isometric on unbounded ends. Let $r_0$ be a number such that the restriction of $f$ to a $r_0$-collar of $\partial M$ is an isometry. 
Fix $\delta = \frac{1}{4}r_0$ and denote by $\pi:f^*T^\delta Y \longrightarrow X$ be the fiber bundle given by the vectors on $f^*T^\delta Y$ whose norm is less or equal to $\delta$. Consider the map
\begin{equation}
\begin{split}
p_f: dom(p_f) \subset &f^*T^\delta Y \longrightarrow Y \\
v_{f(p)} &\longrightarrow exp_{f(p)}(v_{f(p)}).
\end{split}
\end{equation}
If $f$ is not a smooth map, then we will use an $\epsilon$-approximation $f_{\epsilon}$ of it.
\\Notice that outside $\pi^{-1}(M) \subseteq f^*T^\delta Y$ the map $p_f$ could be not defined on all $f^*T^\delta Y$ since $Y$ is not complete. However it is defined on a neighborhood of the $0$-section.
\\Since Lemma 3.3 of \cite{Spes} we know that, on its domain, $p_f$ is a smooth submersion, such that $p_{f}(0_{f(p)}) = f(p)$ and that if $f$ has on $M$ uniformly bounded derivatives with respect to some coordinates $\{x^i\}$ on $X$ and $\{y^j\}$ on $Y$, the same happens to $p_f$ with respect to the fibered coordinates $\{x^i, \mu^j\}$ referred to the frame $\{\frac{\partial}{\partial x^i}\}$. This implies that on
\begin{equation}\label{A_f}
A_f := \{(p,w_{f(p}) \in f^*T^\delta Y \mbox{  such that   } f(p) \subset N \mbox{   and   } \vert w_{f(p)} \vert_h \leq d(f(p), \partial N) \}
\end{equation}
the map $p_f$ is Lipschitz. This can be proved exactly as point 5 of Lemma 3.5 of \cite{Spes} is proved: it is a consequence of Lemma 3.4 of the work of Schick \cite{Schick}.
\\Let $\psi: \mathbb{R} \longrightarrow \mathbb{R}$ be the map defined as follows:
\begin{equation}
\psi(t) := \begin{cases} 0 \mbox{     if      } t \leq 0 \\
1 \mbox{     if      } t \geq \delta \\
P(\frac{t}{\delta})  \mbox{     if      } t \in (0, \delta) \end{cases}
\end{equation}
where $P(s) := -20s^7 + 70s^6 -84s^5 +35s^4$. Thanks to Lemma 2.34 of the work of Eldering \cite{bound}, we can fix $\epsilon < \frac{\delta}{3}$ and $\nu \leq \frac{r_0}{4}$ and to find a map $\psi_{\epsilon}:  \mathbb{R} \longrightarrow \mathbb{R}$ such that 
\begin{itemize}
\item $\psi = \psi_{\epsilon} $ on $\mathbb{R} \setminus (\delta - \nu, \delta + \nu)$,
\item  $d(\psi_\epsilon(t), \psi(t))\leq \epsilon$ for each $t$,
\item $\psi$ is a smooth function on $\mathbb{R} \setminus \{0\}$,
\item $\psi_\epsilon$ is a map of class $C^3$ on $\mathbb{R}$.
\end{itemize}
Then we can define the map $\phi: N \longrightarrow \mathbb{R}$ as 
\begin{equation}
\phi(q) := \begin{cases} 0 \mbox{     if      } q \notin N \\
1 \mbox{     if      } q \in N \setminus \mathcal{N}_X(r_0) \\
\psi_\epsilon(d(q, \partial N))  \mbox{     if      } q \in \mathcal{N}_X(r_0) \end{cases}
\end{equation}
By some easy computations, it is possible to notice that $\phi$ is a Lipschitz function in $C^3(N)$ and it is smooth outside $\partial N$. Indeed, because of Remark \ref{rem1}, if we fix some collar coordinates $(x,t)$ on $\mathcal{N}_Y(r_0)$, then $\phi$ only depends on $t$. 
\\Let $\tilde{p}_f: f^*T^\delta Y \longrightarrow Y$ be the map defined as
\begin{equation}
\tilde{p}_f(v_{f(p)}) := p_{f}(\phi(f(p)) \cdot v_{f(p)}).
\end{equation}
Observe that $\tilde{p}_f$ is defined on all $f^*T^\delta Y$ and it is a submersion. Indeed, on $M$, $\tilde{p}_f$ is a submersion because $\phi \neq 0$ and $p_f$ is a submersion; on $X \setminus M$, $\tilde{p}_f$ is a submersion since $f$ is an isometry.
\begin{prop}
Let $f:(X,g) \longrightarrow (N,h)$ be a uniform map isometric on the unbounded ends. Let $g_{S,f}$ be the Sasaki metric induced by $g$, $f^*h$ and $f^*\nabla^{LC}_h$ on $f^*TN$. The map $\tilde{p}_f: (f^*T^\delta, g_{S,f})  \longrightarrow (Y,h)$ is a R.-N.-Lipschitz map.
\end{prop}
\begin{proof}
Observe that $\tilde{p}_f$ is a Lipschitz map since $p_f$ is Lipschitz on the subset $A_f$ defined in \ref{A_f}, since $f$ is an isometry on the unbounded ends and since $\phi$ is Lipschitz. We have to show that the Fiber Volume is uniformly bounded. Observe that if $q \in Y \setminus N$, then $p_f^{-1}(q)$ is just the fiber over $p_0 = f^{-1}(q)$ of $f^*T^\delta N$. Then, since $f$ is an isometry on the ends, the Fiber Volume is equal to the volume of a ball of radius $\delta$ in $\mathbb{R}^n$.
\\Fix $q$ in $N$ such that $B_\delta(q) \cap \mathcal{N}_Y(\delta) = \emptyset$. Observe that for each $p_0$ in $X$, since $\phi \leq 1$,
\begin{equation}
d(f(p_0), \tilde{p}_f(w_{f(p_0)})) \leq \delta \cdot f^*\phi(p_0) \leq \delta.
\end{equation}
This means that,
\begin{equation}
\tilde{p}_f^{-1}(q) \subset \pi^{-1}f^{-1}(B_\delta(q))
\end{equation}
Then, since $B_\delta(q) \cap \mathcal{N}_Y(r_0) = \emptyset$, we obtain that $f^*\phi$ on $f^{-1}(B_\delta(q))$ is equal to $1$. Moreover, since $\pi(\tilde{p}_f^{-1}(q))$ is contained in $f^{-1}(B_\delta(q))$, and 
\begin{equation}
\emptyset = f^{-1}(B_\delta(q)) \cap f^{-1}( \mathcal{N}_Y(r_0) ) = f^{-1}(B_\delta(q)) \cap \mathcal{N}_X(r_0)
\end{equation}
this means that $\tilde{p}_f$, on the fiber of $q$, is equal to the submersion $p_f$ defined in \cite{Spes}. So the Fiber Volume of $q$ can be computed exactly as in Corollary 4.2 of \cite{Spes} and it is bounded by a constant which only depends on the curvatures of $M$ and $N$.
\bigskip
\\Finally we have to calculate the Fiber Volume of a point $q$ in $N$ such that $B_{\delta}(q) \cap \mathcal{N}_Y(\delta) \neq \emptyset$. Since $\delta = \frac{r_0}{4}$ we obtain that $B_{\delta}(q) \cap N \subset \mathcal{N}_Y(r_0)$,
\\Since $\tilde{p}_f = f \circ \pi$ outside $M$ and since $f_{\vert E_X}$ is an isometry, then $\tilde{p}_f^{-1}(q)$ is contained in $\pi^{-1}(\mathcal{N}_X(r_0))$. Let us define the map 
\begin{equation}
    \begin{split}
    \tilde{t}_f:\pi^{-1}(\mathcal{N}_X(r_0)) &\longrightarrow  M \times N \\
    w_{f(x,t)} &\longrightarrow (x, t, \tilde{p}_f(w_{f(x,t)}))
    \end{split}
\end{equation}
Observe that $\tilde{p}_f = pr_N \circ \tilde{t}_f$ where $pr_N : M \times N \longrightarrow N$ is the projection on the second component.
\\Notice that, since Proposition \ref{compo}, we obtain that 
\begin{equation}
Vol_{\tilde{p}_f}(q) = \int_{X} Vol_{\tilde{t}_f}(p, q) d\mu_X.
\end{equation}
Let us focus on $Vol_{\tilde{t}_f}$. Fix some normal coordinates $\{V, y^i\}$ around $q$ and some collar coordinates $\{U, x^j, t\}$ around $p_0  = (x_0, t_0) = f^{-1}(q)$. Assume that $f(U) \subseteq V$. Then on  $\pi^{-1}(U) \subset f^*T^\delta Y$ we have the fibered coordinates $\{x^i, t, \mu^j\}$ where $\{\mu^j\}$ are the coordinates relative to the frame $\{\frac{\partial}{\partial y^i}\}$. Observe that
\begin{equation}
\tilde{t}_f(x_0, t_0, \mu^j) = (x_0, t_0, \psi_\epsilon(t_0) \cdot \mu^j),
\end{equation}
where $\psi_\epsilon: \mathbb{R} \longrightarrow \mathbb{R}$ is the pa we used to define $\phi$. 
\\Moreover, since the Christoffell symbols of the pullback connection $f^*\nabla^{LC}_h$ vanish in $(x_0, t_0)$, then the volume form of $f^*T^\delta Y$ in $(x_0, t_0, \mu^j)$ is given by
\begin{equation}
Vol_{f^*T^\delta Y}(x_0, t_0, \mu^j) = det(g_{ij}(x_0, t_0))dx^I\wedge dt \wedge d\mu^J.    
\end{equation}
Observe that $det(g_{ij}(x_0, t_0))$ is uniformly bounded since $M$ is an open of bounded geometry. Let us study the Volume form on $M \times N$: this is given by
\begin{equation}
Vol_{M \times N}(x, t, y)  = det(g_{ij}(x,t))\cdot det(h_{rs}(y)) dx^I \wedge dt \wedge dy^J.
\end{equation}
Observe that $\tilde{t}_f$ is a diffeomorphism with its image. This means that its Fiber Volume is null outside the image of $\tilde{t}_f$. Moreover, on $im(\tilde{t}_f)$, the Fiber Volume is given by 
\begin{equation}
    [\tilde{t}_f^{-1}]^*(\frac{Vol_{f^*T^\delta Y}}{\tilde{t}_f^*(Vol_{M \times N})})
\end{equation}
Notice that, in $(x_0, t_0, q)$, we have
\begin{equation}
    \frac{Vol_{f^*T^\delta Y}}{\tilde{t}_f^*(Vol_{M \times N})}(x_0, t_0, \mu^j) = [det(h_{rs}(\psi_\epsilon(t_0) \cdot y))]^{-1} \frac{1}{\psi_\epsilon(t_0)^n}.
\end{equation}
This means that, for each $(x, t, q)$ in $im(\tilde{t}_f)$,
\begin{equation}\label{bound1}
 Vol_{\tilde{t}_f}(x, t, q) \leq C\cdot \frac{1}{\psi_\epsilon(t)^n}.
\end{equation}
Let us decompose $M$ as the union of $\mathcal{N}_X(\frac{3}{10}\delta)$ and $M_{0} := M \setminus \mathcal{N}_X(\frac{3}{10}\delta)$. We are choosing $\frac{3}{10}\delta$ because on $\mathcal{N}_X(\frac{3}{10}\delta)$ the inequality $\phi(x, t) = \psi_\epsilon(t) \leq \frac{1}{2}$ holds. This inequality will be useful later.
\\However we can see the Fiber Volume of $\tilde{p}_f$ as the integral
\begin{equation}
    Vol_{\tilde{p}_f}(q) = \int_{\mathcal{N}_X(\frac{3}{10}\delta)}Vol_{\tilde{t}_f}(p, q) d\mu_M + \int_{M_0} Vol_{\tilde{t}_f}(p , q) d\mu_M.
\end{equation}
Notice that, thanks to (\ref{bound1}), on $M_0$ the Fiber Volume of $\tilde{t}_f$ is uniformly bounded by $C_\psi$ since $\psi_\epsilon(t)$ is bounded from below by $\phi(x, \frac{3}{10}\delta) = \psi_\epsilon(\frac{3}{10}\delta)$. Moreover, since $f$ is uniformly proper, then also $\pi \circ f$ and $\tilde{p}_f$ are uniformly proper and so this means that there is a number $C$ (which does not depends on $q$) such that
\begin{equation}
diam(\pi(\tilde{p}_f^{-1}(q))) \leq C.
\end{equation}
As a consequence of this and of the bounded geometry of $M$ we obtain a constant $K$ such that
\begin{equation}
\mu_M(\pi(\tilde{p}_f^{-1}(q)) \cap M_0) \leq K.
\end{equation}
Finally, observe that $im(\tilde{t}_f) \cap [M\times\{q\}] = \pi(\tilde{p}_f^{-1}(q)) \times \{q\}$ and we obtain that
\begin{equation}
   \int_{M_0} Vol_{\tilde{t}_f}(p , q) d\mu_M \leq \int_{\pi(\tilde{p}_f^{-1}(q))} Vol_{\tilde{t}_f}(p , q) d\mu_M \leq K \cdot C_\psi.
\end{equation}
We only have to prove the boundedness of the integral over $\mathcal{N}_X(\frac{3}{10}\delta).$
\\Recall that if $(x, t, q)$ is not in $im(\tilde{t}_f)$, then Fiber Volume of $Vol_{\tilde{t}_f}(x, t, q) = 0.$ Observe that $im(\tilde{t}_f)$ is given by the elements $(x,t,y)$ of $M \times N$ such that there is a tangent vector $w_{f(x,t)}$ on $f(x,t)$ whose norm is less or equal to $\delta \cdot \phi(t)$ and which satisfies
\begin{equation}
\tilde{p}_f(w_{f(x,t)}) = exp_{f(x,t)}(w_{f(x,t)}) = q = f(x_0, t_0).
\end{equation}
This means that, if $(x,t,y)$ is in $im(\tilde{t}_f)$, then 
\begin{equation}\label{due}
    d_M((x_0,t_0), (x,t)) = d_N(f(x_0, t_0), f(x,t)) =  d_N(q, f(x,t)) \leq \vert\vert w_{f(x,t)} \vert\vert \leq \delta \cdot \psi_\epsilon(t).
\end{equation}
Observe that for each $(x,t)$ in $\mathcal{N}_X(r_0)$
\begin{equation}
    \vert t - t_0 \vert = d_M((x_0,t_0),(x_0,t)) \leq d_M((x_0,t_0), (x,t))
\end{equation}
indeed, if $(x_1, t)$ is a point such that $d_M((x_0,t_0), (x_1,t)) < \vert t - t_0 \vert$ then $d((x_1, t), \partial M) < t$ which is a contradiction with Lemma \ref{lemma0}. Moreover, on $\mathcal{N}_X(\frac{3}{10}\delta)$, we also have that $\psi_\epsilon(t) \leq \frac{1}{2}t$ for each $(x,t)$. Then, because of the inequality (\ref{due}), if $(x,t,y)$ is in $im(\tilde{t}_f) \cap \mathcal{N}_X(\frac{3}{10}\delta)\times N$, then
\begin{equation}\label{bound2}
    \vert t - t_0 \vert \leq \delta \cdot \psi_\epsilon(t) \leq \frac{\delta}{2}t \implies \frac{2}{2+\delta} t_0 \leq t \leq \frac{2}{2 - \delta} t_0
\end{equation}
and so, since $\psi_\epsilon$ is a monotone polynomial of degree $7$, we obtain that
\begin{equation}\label{tre}
 d_M((x_0,t_0), (x,t)) \leq \delta \cdot \psi_\epsilon(t) \leq \delta \cdot \psi_\epsilon(\frac{2}{2 - \delta} t_0) \leq \delta \cdot \frac{2^7}{(2 - \delta)^7} \cdot \psi_\epsilon (t_0).
\end{equation}
This means that $\pi(\tilde{p}_f^{-1}(q))$ is contained on a ball on $M$ of radius $\delta \cdot \frac{2^7}{(2 - \delta)^7} \cdot \psi_\epsilon(t_0)$ where $t_0 = d(q, \partial N)$ and so 
\begin{equation}
    \mu_M(\pi(\tilde{p}_f^{-1}(q))) \leq J_1 \cdot \psi_\epsilon(t_0)^n.
\end{equation}
Moreover, if $(x,t,y)$ is an element of $im(\tilde{t}_f) \cap \mathcal{N}_X(\frac{3}{10}\delta)\times N$, thanks to formula (\ref{bound2}) we obtain
\begin{equation}
    \frac{1}{\psi_\epsilon(t)^n} \leq \frac{1}{\psi_\epsilon(\frac{2}{2+\delta}t_0)^n} = \frac{(2 + \delta)^{7n}}{2^{7n}} \cdot \frac{1}{\psi_\epsilon(t_0)^n} \leq J_2 \cdot \frac{1}{\psi_\epsilon(t_0)^n}.
\end{equation}
So we conclude as follows: given $q = (y_0, t_0)$ in $N$ we obtain 
\begin{equation}
\int_{\mathcal{N}_X(\frac{3}{10}\delta)}Vol_{\tilde{t}_f}(p, q) d\mu_M \leq J_1 \cdot \phi(t_0)^n \cdot  J_2 \cdot \frac{1}{\phi(t_0)^n} \leq B.
\end{equation}
\end{proof}
\subsection{The pullback operator}
Let $\omega$ be the Mathai-Quillen-Thom form of $f^*TY$ with support contained in $f^*T^\delta Y$ defined in subsection \ref{form}. Then, for each uniform map isometric on the unbounded ends, we can define the pullback operator
\begin{equation}
    T_f(\alpha) := \int_{B^\delta} \tilde{p}_f^*\alpha \wedge \omega.
\end{equation}
\begin{prop}\label{T_f}
The operator $T_f$ satisfies the following properties:
\begin{enumerate}
    \item $T_f$ is $\mathcal{L}^2$-bounded,
    \item $T_f(dom(d_{min})) \subset dom(d_{min})$ and $T_f(dom(d_{max})) \subset dom(d_{max})$,
    \item $T_f$ and $d$ commute.
\end{enumerate}
\end{prop}
\begin{proof}
Let $\chi_{N}$ and $\chi_{Y \setminus N}$ be the characteristic functions relative to $M$ and to $X \setminus M$. Observe that if $\alpha$ is a $\mathcal{L}^2$-form, then $\chi_N \alpha$ and $\chi_{Y \setminus N}\alpha$ are  $\mathcal{L}^2$-forms and
\begin{equation}
\vert \vert \alpha \vert \vert^2 = \vert \vert \chi_N \alpha \vert \vert^2 + \vert \vert \chi_{ Y \setminus N}\alpha \vert \vert. 
\end{equation}
Observe that
\begin{equation}
T_f(\chi_{Y \setminus N}\alpha) = \chi_{X \setminus M} f^*\alpha   
\end{equation}
and so, since $f$ on $X \setminus M$ is an isometry, we have that
\begin{equation}
\vert \vert T_f(\chi_{Y \setminus N}\alpha) \vert \vert^2 = \vert \vert \chi_{ Y \setminus N}\alpha \vert \vert^2.
\end{equation}
Let us focus on $\chi_N \alpha$. Observe that
\begin{equation}
T_f (\chi_N \alpha) = pr_{M\star} \circ e_\omega \circ \tilde{p}_f^* (\chi_N \alpha)
\end{equation}
where $e_\omega(\beta) := \beta \wedge \omega$ and $pr_{M\star}$ is the integration along the fiber of $f^*(T^\delta N)$.
\\We already know that $\tilde{p}_f^*$ is a R.-N.-Lipschitz map and so this means that $\tilde{p}_f^*$ is a $\mathcal{L}^2$-bounded operator. Moreover, since Proposition 3.4. of \cite{Spes}, we also know that $pr_{M\star}$  is $\mathcal{L}^2$-bounded. The main problem is $e_\omega$: indeed, since the norm $\vert \omega \vert_p$ could be not uniformly bounded in $p$ in $X$, the operator $e_\omega$ can be not $\mathcal{L}^2$-bounded. However, observe that on $\pi^{-1}(M)$ the norm $\vert \omega \vert_p$ actually is uniformly bounded since $M$ and $N$ have bounded geometry. Then, as a consequence of Proposition 4.4, if we restrict $e_\omega: \mathcal{L}^2(\overline{M}) \longrightarrow \mathcal{L}^2(\overline{M})$, we obtain a $\mathcal{L}^2$-bounded operator. Moreover we also have that $\tilde{p}_f^* (\chi_N \alpha)$ has support contained in $\pi^{-1}(\overline{M})$ and so this means that
\begin{equation}
\vert \vert T_f(\chi_{N}\alpha) \vert \vert^2 \leq C \cdot \vert \vert \chi_{N}\alpha \vert \vert^2.
\end{equation}
Observe that $T_f(\chi_{N}\alpha) = 0$ on $X \setminus M$ and $T_f(\chi_{Y \setminus N}\alpha) = 0$ on $M$. Then we obtain
\begin{equation}
\begin{split}
\vert \vert T_f\alpha \vert \vert^2 &=  \vert \vert T_f\chi_N \alpha \vert \vert^2 + \vert \vert T_f\chi_{Y \setminus N}\alpha \vert \vert^2       \\
&\leq K \cdot \vert \vert \chi_{N}\alpha \vert \vert^2 + \vert \vert \chi_{Y \setminus N}\alpha \vert \vert^2 \\
&\leq \mbox{max}(K, 1) (\vert \vert \chi_{N}\alpha \vert \vert^2 + \vert \vert \chi_{Y \setminus N}\alpha \vert \vert^2)\\
&=  \mbox{max}(K, 1) \vert \vert \alpha \vert \vert^2.
\end{split}
\end{equation}
So we proved 1.
\\Then we will focus on proving that $T_f(dom(d_{maxY})) \subseteq (dom(d_{maxX}))$. 
\\Notice that, since $\tilde{p}_f$ is a map of class $C^3$, the pullback along $\tilde{p}_f:f^*T^\delta N \longrightarrow N$ sends $\Omega^*(N)$ on $C^2(\Lambda^*(f^*T^\delta N))$, which is the space of differential forms of class $C^2$ over $f^*T^\delta N$. Observe that on this space the exterior derivative of a form $\alpha$ is defined exactly as it is defined for smooth forms (pag. 549 of \cite{lang}) and all the properties of $d$ are the same. In particular $\tilde{p}^*_f$ and the exterior derivative operator commute. 
\\The operator $e_\omega(\alpha) := \alpha \wedge \omega$ also commute with the exterior differential and sends $C^2(\Lambda^*(f^*T^\delta N))$ in itself. Finally,  if the support of $\alpha$ is vertically compact, the integration along the fibers $\pi_\star$ of $f^*T^\delta N$ also satisfies $\pi_\star d \alpha = d \pi_\star \alpha$ (the proof is the same of the classical one: see for example Proposition 6.14.1. of \cite{bottu}) and, moreover, if $\alpha$ is a $C^k$-form, the same holds for $\pi_\star \alpha$.
\\So we obtain that if $\alpha$ is a smooth form on $Y$, then $T_f\alpha$ is a $C^2$-form on $X$ and $T_f d \alpha = dT_f\alpha$. Moreover, since $\tilde{p}_f$ is uniformly proper, if $\alpha$ is an element of $\Omega^*_c(N)$, then $T_f\alpha$ is compactly supported. Finally also observe that outside $\partial M$, which has null measure on $X$, $T_f\alpha$ is a smooth differential form.
\\
Observe that $T_f^\dagger =  \tilde{p}_{f\star} \circ r_\omega \circ \pi^*$, where $\tilde{p}_{f\star}$ is the integration along the fibers of $\tilde{p}_{f}$ and $r_\omega(\beta) = \omega \wedge \beta$. Notice that since $\tilde{p}_{f}$ is a $C^3$ map, then for each $q$ in $N$ and for each $p$ in $f^*T^\delta N$ there are a couple of $C^3$-charts $(U, x^1, \ldots., x^n, y^1, \ldots, y^n)$ and $(V, y^1, \ldots, y^n)$ such that $\tilde{p}_{f}(x^1, \ldots., x^n, y^1, \ldots, y^n) = (y^1, \ldots, y^n)$. This implies, by a partition of unity argument, that the integral along the fibers of $\tilde{p}_{f}$ of a smooth form is $C^3$-form. So $d T_f^\dagger \gamma $ is a compactly supported $C^3$-form. Then we can conclude by applying Corollary \ref{corollary1}.
\end{proof}
\subsection{The isomorphisms induced by the pullback}
Let $(X,g)$ and $(Y,h)$ be two manifolds of bounded geometry with unbounded ends and let $f: (M,g) \longrightarrow (N,h)$ be a uniform homotopy equivalence isometric on the ends. In this section we introduce a couple of operators $y: \mathcal{L}^2(Y) \longrightarrow \mathcal{L}^2(Y)$ and $z: \mathcal{L}^2(X) \longrightarrow \mathcal{L}^2(X)$ such that $y(dom(d_{max}) \subseteq dom(d_{max})$, $y(dom(d_{min}) \subseteq dom(d_{min})$, $z(dom(d_{max}) \subseteq dom(d_{max})$, $z(dom(d_{min}) \subseteq dom(d_{min})$ and
\begin{equation}
    1 \pm T_f^\dagger T_f = dy + yd \mbox{    and    } 1 \pm T_fT_f^\dagger = dz + zd
\end{equation}
on the minimal or also in the maximal domain of the exterior derivative operator. In this formula we have a $+$ if $f$ reverses the orientations on $-$ otherwise.
\\
\\In order to define these operators, we need a metric on the bundle $f^*(TY)\oplus f^*(TY)$ over $X$. We consider the generalized Sasaki metric $g_S$ induced by $g$, $h$ and $\nabla^{LC}_h$. In general, even if $f_1 = f_2 = f$, we will denote the bundle $f^*(TY)\oplus f^*(TY)$ by $f^*_1(TY)\oplus f^*_2(TY)$. Moreover, given $i=1,2$, we denote by $pr_i: f^*_1(TY)\oplus f^*_2(TY) \longrightarrow f^*(TY)$ the projection on the $i$-th component, i.e. $pr_i(w_{f_1(p)}\oplus w_{f_2(p)}) := w_{f_i(p)}$ and so we obtain the maps $\tilde{p}_{f,i} := \tilde{p}_f \circ pr_i$ and $\pi_i := \pi \circ pr_i.$
\\Finally we denote by
\begin{equation}
    \mathcal{B}:= \{w_{f_1(p)}\oplus w_{f_2(p)} \in f^*_1(TY)\oplus f^*_2(TY) s.t. \vert w_{f_1(p)}\vert_h \leq \delta, \vert w_{f_2(p)}\vert_h \leq \delta \}.
\end{equation}
\begin{lem}
Assume that $(X,g)$ and $(Y,h)$ are two Riemannian manifolds and let $f:(X,g) \longrightarrow (Y,h)$ be a smooth uniform map isometric on the unbounded ends of $X$. Then there are two $\mathcal{L}^2$-bounded operators $y_0: \mathcal{L}^2(Y) \longrightarrow \mathcal{L}^2(Y)$ and $z_0: \mathcal{L}^2(X) \longrightarrow \mathcal{L}^2(X)$ such that they preserve the minimal and the maximal domains of the exterior derivative operators
\begin{equation}
   \tilde{p}_{f,2}^* - \tilde{p}_{f,1}^* = dy_0 + y_0d
\end{equation}
and finally
\begin{equation}
   \pi_{2}^* - \pi_{1}^* = dz_0 + z_0d
\end{equation}
on the maximal domains of exterior derivative operators.
\end{lem}
\begin{proof}
Let $\int_{0,\mathcal{L}}^1 : \Omega^*(\mathcal{B} \times [0,1]) \longrightarrow \Omega^*(\mathcal{B})$ be the operator defined in Lemma 4.13 of \cite{Spes}. It is defined as follows: if $\alpha$ is a $0$-form with respect to $[0,1]$, then
		\begin{equation}
		\int_{0,\mathcal{L}}^{1}\alpha := \int_{0,\mathcal{L}}^{1} g(x,t)p^*\omega =  0
		\end{equation}
	and, if $\alpha$ is a $1$-form with respect to $[0,1]$,
		\begin{equation}
		\int_{0,\mathcal{L}}^{1}\alpha := (\int_{0}^{1} f(x,t)dt)\omega.
		\end{equation}
		This is an $\mathcal{L}^2$-bounded operator and it sends compactly supported $C^1$-forms in   $C^1_c(\Lambda^*(T^*\mathcal{B})).$ Moreover, if $\alpha$ is a $C^1$-form on $\mathcal{B} \times [0,1]$, then
		\begin{equation}
		    j_1^*\alpha - j_0^*\alpha = d\int_{0,\mathcal{L}}^{1}\alpha + \int_{0,\mathcal{L}}^{1}d\alpha,
		\end{equation}
		where $j_i: \mathcal{B} \longrightarrow \mathcal{B} \times [0,1]$ is given by $j_i(x) := (x,i).$
		\\Let $\phi_1$ and $\phi_0$ be two R.-N.-Lipschitz maps which are uniformly homotopic with a uniformly proper, R.-N.-Lipschitz homotopy $H$. Then, as a consequence of Corollary \ref{corollary1}, the operator $\int_{0,\mathcal{L}}^{1} \circ H$ is an $\mathcal{L}^2$-bounded operator which preserves the minimal and the maximal domains of the exterior derivative operators and, on the maximal domain, we have
		\begin{equation}
		    \phi_1^* - \phi_0^* = d (\int_{0,\mathcal{L}}^{1} \circ H) + (\int_{0,\mathcal{L}}^{1} \circ H) d
		\end{equation}
	So, in order to conclude the proof it is sufficient to find a couple of uniformly proper, R.-N.-Lipschitz homotopies.
	\\Let $A: (\mathcal{B} \times [0,1], g_S + dt) \longrightarrow (f^*T^\delta Y, g_f)$ be defined as
\begin{equation}
    A(w_{f(p)} \oplus v_{f(p)}, s) := s\cdot w_{f(p)} + (1-s)v_{f(p)}.
\end{equation}
It is an easy exercise to prove that $A$ is a uniformly proper R.-N.-Lipschitz map: a proof of this fact can be also found in the work of the author \cite{Spes2}. Then we can conclude by observing that $\tilde{p}_f \circ A$ is a homotopy between $\tilde{p}_{f_1}$ and $\tilde{p}_{f_2}$ and $\pi \circ A$ is a homotopy between $\pi_1$ and $\pi_2$.
\end{proof}
We also need the following Lemma: the proof is very similar to the same given in \cite{Spes2}, however, for the sake of completeness, here it is proved.
\begin{lem}\label{y and Y}
		Let $(Y,g)$ be a manifold of bounded geometry with some possibly unbounded ends. Fix $\tilde{p}_{id}: T^\delta Y \longrightarrow Y$ and let $\omega$ be a Thom form of the bundle $\pi:TY \longrightarrow Y$, where $\pi(v_p) = p$, such that $supp(\omega) \subset T^\delta Y$. Then for all $q$ in $N$
		\begin{equation}
		\int_{F_q} \omega = 1
		\end{equation}
		where $F_q$ is the fiber of  $\tilde{p}_{id}$.
	\end{lem}
	\begin{proof}
		For all $q$ in $N$ the fiber $F_q$ is an oriented compact submanifold with boundary. The same also holds for $B^\delta_q$ which is the fiber of the projection $\pi: T^\delta N \longrightarrow N$ defined as $\pi(v_q) := q$.
		\\Define $H:T^\delta N \times [0,1] \longrightarrow N$ as
		\begin{equation}
		H(v_p,s) = \tilde{p}_{id}(s\cdot v_p).
		\end{equation}
		Since $H$ is a proper submersion, the fiber along $H$ given by $F_{H,q}$ is submanifold of $T^\delta N \times [0,1]$. Its boundary, in particular is
		\begin{equation}
		\partial F_{H,q} = B^\delta_q \times \{0\} \sqcup F_q \times \{1\} \cup A
		\end{equation}
		where $A$ is contained in
		\begin{equation}
			S^\delta N := \{v_{p} \in TN \vert \vert v_p\vert = \delta\}.
		\end{equation}
		Then, if $\omega$ is a Thom form of $TN$ whose support is contained in $T^\delta N$, then
		\begin{equation}\label{nano}
		0 = \int_{F_{H_q}} d \omega = d \int_{F_{H_q}} \omega + \int_{\partial F_{H_q}} \omega.
		\end{equation}
		Observe that $\omega$ is a $k$-form and $dim(F_{H_q})= k+1$. Then the first integral on the right side of \ref{nano} is $0$. Moreover, we obtain that $\omega$ is null on $A$, and so
		\begin{equation}
	\int_A \omega = 0.
		\end{equation}
	 Then, by the equality (\ref{nano}),
		\begin{equation}
		0 = \mp \int_{B^\delta} \omega \pm \int_{F_q}\omega.
		\end{equation}
	and we conclude.
	\end{proof}
\begin{lem}\label{lemma1}
Let $(X,g)$ and $(Y,h)$ be two manifolds of bounded geometry with unbounded ends and let $f: (M,g) \longrightarrow (N,h)$ be a uniform homotopy equivalence isometric on the ends. Then there are a couple of operators $y: \mathcal{L}^2(Y) \longrightarrow \mathcal{L}^2(Y)$ and $z: \mathcal{L}^2(X) \longrightarrow \mathcal{L}^2(X)$ such that they preserve the minimal and maximal domain of the exterior derivative operators and
\begin{equation}
    1 \pm T_f^\dagger T_f = dy + yd \mbox{    and    } 1 \pm T_fT_f^\dagger = dz + zd
\end{equation}
on the maximal domain of the exterior derivative operator. In this formula there is a $+$ if $f$ reverses the orientations and $-$ otherwise.
\end{lem}	
\begin{proof}
We will start by $z$. We have the following diagrams.
	\begin{equation}
	\xymatrix{
& \mathcal{B}  \ar[d]^{\tilde{P}_{f,2}} \ar[r]^{\tilde{P}_{f,1}} & f^*_2(T^\delta Y) \ar[d]^{\tilde{p}_{f}}&\\
& f^*_1(T^\delta Y) \ar[r]^{\tilde{p}_{f}} & Y}
\end{equation}
where $\tilde{P}_{f_1}$ and $\tilde{P}_{f_2}$ are the bundle maps induced by $p_f$ and 
\begin{equation}
\xymatrix{
& \mathcal{B}  \ar[d]^{\Pi_{2}} \ar[r]^{\Pi_{1}} & f^*_2(T^\delta Y) \ar[d]^{\pi}&\\
& f^*_1(T^\delta Y) \ar[r]^{\pi} & X}
	     \end{equation}
	 where $\Pi_{1}$ and $\Pi_{2}$ are the bundle maps induced by $\pi$.
\\Fix $\alpha$ and $\beta$ in $\Omega^k_c(X)$ for some $k$ in $\mathbb{N}$. Then
\begin{equation}
\begin{split}
    \langle T_fT_f^\dagger \alpha; \beta \rangle_X &= \langle T_f^\dagger \alpha; \tau T_f^\dagger \tau \beta \rangle_X \\
    &= \int_{N} (\int_{F_2}\omega_2 \wedge \pi^*_2\alpha) \wedge (\int_{F_1}\omega_1 \wedge \pi^*_1 \tau \beta) \\
    &= deg(f) \int_{f^*_1TN} (\int_{F_2}\omega_2 \wedge \pi^*_2\alpha) \wedge \omega_1 \wedge \pi^*_1 \tau \beta \\
    &= deg(f) (-1)^{j(n+j)} \int_{f^*_1TN} \omega_1 \wedge \pi^*_1 \tau \beta \wedge (\int_{F_2}\omega_2 \wedge \pi^*_2\alpha) \\
    &= deg(f) (-1)^{j(n+j)} \int_{\mathcal{B}} \omega_1 \wedge \pi^*_1 \tau \beta \wedge \omega_2 \wedge \pi^*_2\alpha\\
    &= deg(f) \int_{\mathcal{B}} \pi^*_2\alpha \wedge \pi^*_1 \tau \beta \wedge \omega_1  \wedge \omega_2.
\end{split}
\end{equation}
Then, as a consequence of $\pi_2^* = \pi_1^* + dz_0 + z_0 d$, we obtain
\begin{equation}
    \begin{split}
     \langle T_fT_f^\dagger \alpha; \beta \rangle_X &= deg(f) \int_{\mathcal{B}} \pi^*_1\alpha \wedge \pi^*_1 \tau \beta \wedge \omega_1  \wedge \omega_2\\
     &+ deg(f) \int_{\mathcal{B}} (dz_0 + z_0 d) \alpha \wedge \pi^*_1 \tau \beta \wedge \omega_1  \wedge \omega_2 \\
     &= \langle deg(f)\cdot 1 \alpha; \beta \rangle_X + \langle [d(-z) + (-z)d]\alpha; \beta \rangle_X
    \end{split}
\end{equation}
where $z := - deg(f) \cdot \pi_\star \circ e_\omega \circ \Pi_{2\star} \circ e_{\Pi_2^*\omega} \circ z_0.$
\\Observe that $z$ is a composition of $\mathcal{L}^2$-bounded operators and $z(\Omega^*_c(X)) \subseteq C^2_c(\Lambda^*(T^*X))$. Moreover we also have that
\begin{equation}
z^\dagger = - deg(f) z_0^\dagger \circ r_{\Pi_2^*\omega} \circ \Pi_2^* \circ r_\omega \circ \pi^*
\end{equation}
is an $\mathcal{L}^2$-bounded operators and $-z^\dagger(\Omega^*_c(Y)) \subseteq C^2_c(\Lambda^*(T^*Y))$. So by Corollary \ref{corollary1}, we have that $z$ preserves the minimal and the maximal domain of $d$ and 
\begin{equation}
    1 \pm T_fT_f^\dagger = dz + zd
\end{equation}
where we have a $-$ if $f$ preserves the orientations and a $+$ otherwise.
\\
\\
\\Let us focus on the operator $y$. Given $\alpha$ and $\beta$ in $\Omega^*_c(Y)$, we obtain that
\begin{equation}
		\begin{split}
		\langle T^\dagger_fT_f \alpha, \beta \rangle &=  \langle T_f \alpha, \tau T_f \tau \beta \rangle \\
		&= \int_{M} (\int_{B^\delta_2}\tilde{p}_{f,2}^*\alpha \wedge \omega_2) \wedge (\int_{B^\delta_1}\tilde{p}_{f,1}^*\tau \beta \wedge \omega_1) \\
		&= \int_{f^*(T^\delta Y)_1} (\int_{B^\delta_2}\tilde{p}_{f,2}^*\alpha \wedge \omega_2) \wedge \tilde{p}_{f,1}^*\tau \beta \wedge \omega_1 \\
		&= (-1)^{(n + j)j} \int_{f^*(T^\delta Y)_1} \tilde{p}_{f,1}^*\tau \beta \wedge \omega_1 \wedge (\int_{B^\delta_2}\tilde{p}_{f,2}^*\alpha \wedge \omega_2)  \\
		&= (-1)^{(n + j)j} \int_{\mathcal{B}} \tilde{p}_{f,1}^*\tau \beta \wedge \omega_1 \wedge \tilde{p}_{f,2}^*\alpha \wedge \omega_2  \\
		&= (-1)^{(n + j)j}(-1)^{(n + j)j} \int_{\mathcal{B}} \tilde{p}_{f,2}^*\alpha \wedge \tilde{p}_{f,1}^*\tau \beta \wedge \omega_1 \wedge \omega_2  \\
		&= \int_{\mathcal{B}} \tilde{p}_{f,2}^*\alpha \wedge \tilde{p}_{f,1}^*\tau \beta \wedge \omega_1  \wedge \omega_2
		\end{split}
		\end{equation}
Then, if we consider the operator $y_0$ defined before, we obtain that
\begin{equation}
		\begin{split}
		\langle T^\dagger_fT_f \alpha, \beta \rangle &= \int_{\mathcal{B}} \tilde{p}_{f,1}^*(\alpha \wedge \tau \beta) \wedge \omega_1  \wedge \omega_2 \\
		&= \int_{\mathcal{B}} (dy_0 + y_0d) \alpha \wedge \tilde{p}_{f,1}^*\tau \beta \wedge \omega_1  \wedge \omega_2\\
		&= deg(f) \int_{TN} \tilde{p_{id}}^* (\alpha \wedge \tau \beta) \wedge \omega \\
		&+ \langle (dy + yd) \alpha, \beta \rangle_Y\\
		&= \langle deg(f)\cdot 1(\alpha); \beta \rangle_Y +  \langle (dy + yd) \alpha, \beta \rangle_Y
		\end{split}
		\end{equation}
where $ y := p_{f,\star} \circ e_{\omega} \circ \Pi_{1,\star} \circ e_{\Pi^*_1 \omega} \circ y_0$. Notice that $y$ is a composition of $\mathcal{L}^2$-bounded operators and $y(\Omega^*_c(Y)) \subseteq C^2_c(\Lambda^*(T^*Y))$. Moreover
\begin{equation}
y^\dagger = y_0^\dagger \circ r_{\Pi^*_1 \omega} \circ \Pi_{1}^* \circ r_\omega \circ p_f^* 
\end{equation}
and so $y^\dagger (\Omega^*_c(X)) \subseteq C^2_c(\Lambda^*(T^*X))$. So, as a consequence of Corollary \ref{corollary1}, we obtain that $y$ preserves the maximal and the minimal domain of $d$ and
\begin{equation}
1 \pm  T^\dagger_fT_f = dy + yd.
\end{equation}
\end{proof}
\begin{prop}\label{proposition1}
Let $f:(M,g) \longrightarrow (N,h)$ be a uniform homotopy equivalence quasi-isometric on the unbounded ends. Then, for each $k$ in $\mathbb{N}$, the operators $T_f$ and $T^\dagger_f$ induce the following isomorphisms
\begin{equation}
    \begin{split}
    H^k_{2, max}(M,g) &\cong H^k_{2, max}(N,h) \\
    \overline{H}^k_{2, max}(M,g) &\cong \overline{H}^k_{2, max}(N,h)\\
    H^k_{2, min}(M,g) &\cong H^k_{2, min}(N,h) \\
    \overline{H}^k_{2, min}(M,g) &\cong \overline{H}^k_{2, min}(N,h).
    \end{split}
\end{equation}
\end{prop}
\begin{proof}
Thanks to Lemma \ref{ban}, we can assume, without loss of generality, that $f$ is an isometry on th unbounded ends. Then the proof immediately follows by Lemma \ref{lemma1}.
\end{proof}
\begin{rem}
The morphisms induced by $T_f$ in maximal and minimal $\mathcal{L}^2$-cohomology, do not depend on the choice of the smooth $\varepsilon$-approximation $f_\varepsilon.$ The proof of this fact is essentially the same given in points 2 and 3 of Proposition 4.17 of \cite{Spes}. 
\end{rem}
\section{Consequences}
\subsection{Mapping cone}
In this subsection, given a Riemannian manifold $(M, g)$, we will denote by $\Omega^*_{m \backslash \mathcal{M}}(M,g):= dom(d_{M, min \backslash max})$ and by $d_{M, m \backslash \mathcal{M}} := d_{M, min \backslash max}$.
\\Let $f:(M,g) \longrightarrow (N,h)$ a uniform map quasi-isometric on the unbounded ends. It is not required that $f$ is a uniform homotopy equivalence. Thanks to Proposition \ref{T_f}, we know that there is a $\mathcal{L}^2$-bounded operator $T_f: \mathcal{L}^2(N,h) \longrightarrow  \mathcal{L}^2(M,g)$ such that $T_f(\Omega^*_{m \backslash \mathcal{M}}(M,g)) \subseteq \Omega^*_{m \backslash \mathcal{M}}(N, h)$. In this subsection we will define the $L^2$-mapping cone of a map $f$ and we will see some properties of this cone.
\begin{defn}
Let $f: (M, g) \longrightarrow (N,h)$ be a uniform map quasi-isometric on the unbounded ends between two  Riemannian manifold and let us denote by $(\Omega^*_{2, min \backslash max} (f), d_{f, min \backslash max})$ the cochain complexes 
\begin{equation}
    0 \rightarrow \Omega^0_{2, m \backslash \mathcal{M}} (f) \xrightarrow{d_{f, m \backslash \mathcal{M}, 0}} \Omega^1_{2, m \backslash \mathcal{M}} (f) \xrightarrow{d_{f, m \backslash \mathcal{M}, 1}} \Omega^3_{2, m \backslash \mathcal{M}} (f) \xrightarrow{d_{f, m \backslash \mathcal{M}, 2}} \ldots
\end{equation}
where $\Omega^*_{2, m \backslash \mathcal{M}} (f) :=  \Omega^*_{m \backslash \mathcal{M}}(N,h) \oplus \Omega^{*-1}_{m \backslash \mathcal{M}}(M, g)$ and $d_{f, m \backslash \mathcal{M}}(\alpha, \beta) := (-d_{N, m \backslash \mathcal{M}} \alpha; T_f\alpha - d_{M, m \backslash \mathcal{M}}\beta).$
\end{defn}
\begin{defn}
The \textbf{$k$-th group of the $L^2$-mapping cone of $f$} is the cohomology group of the $L^2$-mapping cone, i.e.
\begin{equation}
H^k_{2, m \backslash \mathcal{M}}(f) := \frac{ker(d_{f, m \backslash \mathcal{M}, k})}{im(d_{f, m \backslash \mathcal{M}, k})}.
\end{equation}
The \textbf{reduced $k$-th group of the $L^2$-mapping cone of $f$} is the group defined as
\begin{equation}
\overline{H}^k_{2, m \backslash \mathcal{M}}(f) := \frac{ker(d_{f, m \backslash \mathcal{M}, k})}{\overline{im(d_{f, m \backslash \mathcal{M}, k})}}.
\end{equation}
\end{defn}
Exactly as the mapping cone in the de Rham case, we have a short exact sequence
\begin{equation}
0 \rightarrow \Omega^{*-1}_{2, m \backslash \mathcal{M}}(M,g) \xrightarrow{A} \Omega^*_{2, m \backslash \mathcal{M}} (f) \xrightarrow{B} \Omega^*_{2, m \backslash \mathcal{M}}(N,h) \rightarrow 0,
\end{equation}
where $A(\omega) := (0, \omega)$ and $B(\alpha, \omega) := \alpha$. This sequence induces a long exact sequence on cohomology, and so
\begin{equation}
    0 \rightarrow H^0_{2, m \backslash \mathcal{M}}(f) \rightarrow H^0_{2, m \backslash \mathcal{M}}(M,g) \xrightarrow{\delta} H^0_{2, m \backslash \mathcal{M}}(N,h) \rightarrow H^1_{2, m \backslash \mathcal{M}}(f) \rightarrow \ldots
\end{equation}
where $\delta$ is the connecting homomorphism.
\\Following the same proof given at pag. 78 of the book of Bott and Tu \cite{bottu}, we obtain that $\delta[\omega] := [T_f \omega]$. So we obtain the following Proposition.
\begin{prop}\label{cone}
Let $f:(M,g) \longrightarrow (N,h)$ be a uniform map quasi-isometric on the unbounded ends of $M$ and $N$. Then, the two following statements are equivalent:
\begin{enumerate}
    \item the morphism induced by $T_f$ on $L^2$-cohomology is an isomorphism,
    \item all the cohomology groups of $\Omega^*_{2, m \backslash \mathcal{M}} (f)$ are null.
\end{enumerate}
\end{prop}
\begin{proof}
This is a classical proof which holds for each cochain morphism $T_f :A \longrightarrow B$ between cochain complexes on an additive category. 
\end{proof}
\begin{rem}
The author used a \textit{bounded geometry version} of Proposition \ref{cone} in \cite{Spes2} in order to prove the invariance of the Roe Index of signature operator of a manifold of bounded geometry under uniform homotopy equivalences which preserve the orientations. As a consequence of Proposition \ref{cone}, it could be interesting to prove to generalize this result in a larger setting, for example in the case of complete Riemannian manifolds with uniform homotopy equivalence which are quasi-isometric on the unbounded ends.
\end{rem}
In the reduced case we also obtain a long sequence
\begin{equation}
    0 \rightarrow \overline{H}^0_{2, m \backslash \mathcal{M}}(f) \xrightarrow{B^*} \overline{H}^0_{2, m \backslash \mathcal{M}}(M,g) \xrightarrow{T_f} \overline{H}^0_{2, m \backslash \mathcal{M}}(N,h) \xrightarrow{A^*} \overline{H}^1_{2, m \backslash \mathcal{M}}(f) \rightarrow \ldots
\end{equation}
but it is not exact this time. Indeed it is exact only on $\overline{H}^k_{2, m \backslash \mathcal{M}}(f)$ and on $\overline{H}^k_{2, m \backslash \mathcal{M}}(M,g)$ while on $\overline{H}^k_{2, m \backslash \mathcal{M}}(N,h)$ is just \textit{weakly exact} which means that $ker(T_f) = \overline{im(B^*)}.$
\subsection{Uniform homotopy invariance of signature}
Let $(M,g)$ be a complete Riemannian manifold. Recall that in this case there is only one closure for the operator $d$. This means that the maximal and minimal $L^2$-cohomology groups coincide, both the reduced and the unreduced ones. In next pages we will denote by $d$ the unique closed extension of the exterior derivative operator.
\\In this subsection we introduce the $L^2$-signature $\sigma_M$ of a manifold $(M, g)$ with $dim(M) = 4k$ such that $\overline{H}^{2k}_2(M,g)$ is finite dimensional. This $L^2$-signature is the signature of a pairing defined on $\overline{H}^{2k}_{2}(M,g)$. Consider the operator $d + d^*$. This operator switches the eigenspaces of the chiral operator $\tau$, so it is possible to define the $L^2$-signature operator $(d + d^*)^+$ as the restriction of $d + d^*$ to the $+1$-eigenspace of $\tau$.
\\We obtain that $(d + d^*)^+$ is a Fredholm operator and the its index equals the $L^2$-signature.
The definitions and the proofs of all these facts can be found in the work of Bei \cite{Bei3}, in particular at pag. 19-21.
\begin{prop}
Let $f:(M,g) \longrightarrow (N,h)$ be a uniform homotopy equivalence quasi-isometric on the unbounded ends. In this subsection we will prove that
\begin{enumerate}
\item $T_f$ well-behaves with respect to the pairings on $M$ and $N$, i.e. $\langle T_f [\alpha], T_f \beta] \rangle_M = deg(f) \cdot \langle [\alpha], [\beta] \rangle_N$,
\item from the previous point we obtain that $\sigma_M = deg(f) \cdot \sigma_N$ and the index of the $L^2$-signature operator is an invariant under uniform homotopy equivalences quasi-isometric on the unbounded ends which preserve the orientations.
\end{enumerate}
\end{prop}
\begin{proof}
In order to prove the first point, we need to introduce the pairing $\langle \cdot ; \cdot \rangle_M$. This is defined in pag. 19 of \cite{Bei3} as
\begin{equation}\label{pairing}
	\begin{split}
	  \langle \cdot , \cdot \rangle_M : &\overline{H}^i_{2}(M, g) \times \overline{H}^i_{2}(M, g) \longrightarrow \mathbb{R} \\
	    &([\eta], [\omega]) \longrightarrow \int_{M} \eta \wedge \omega
	    \end{split}
	\end{equation}
Let $id$ be the identity map on $N$ and let $\tilde{p}_{id}: T^\delta N \longrightarrow N$ be the submersion related to $id$. Because of Lemma \ref{lemma1}, we know that this is a R.-N.-Lipschitz map and so $T_{id}$ is a $L^2$-bounded operator. Let us denote by $K := \int_{0\mathcal{L}}^1 \circ \tilde{p}_{h}^*$, where $\tilde{p}_h: T^\delta N \times [0,1] \longrightarrow N$ is defined as $\tilde{p}_{h}(w_{p}, s) := \tilde{p}_{id}(s\cdot w_p)$. By applying the same proof of Lemma 4.11 of \cite{Spes}, we obtain that $\tilde{p}_{h}$ is a R.-N.-Lipschitz map and so $K$ is an $L^2$-bounded operator such that 
\begin{equation}
T_{id} - Id = dK + Kd    
\end{equation}
for each smooth form $\alpha$. Then, as a consequence of Lemma \ref{lemma1}, we obtain that $T_{id}$ and $K$ both preserve the minimal and the maximal domain of $d$.
\\Observe that the operator $T_f = f^* \circ T_{id}$: this directly follows by the definition of $T_f$ and by $\int_{F_0'} F^*\alpha = f^* \int_{F_0} \alpha$. Here $F^*$ is the bundle map induced by $f$, and $F_0$ and $F_0'$ are the fibers of $TN$ and of $f^*TN$. Observe that $T_{id} \alpha $ is a smooth form outside a subset of null measure (that is $\partial N$) and the same happens to $f^*T_{id} \alpha$. 
\\Then we can easily conclude as follows
	\begin{equation}
	    \begin{split}
	   \langle [T_f \alpha] , [T_f \beta] \rangle_M &= \int_M T_f\alpha \wedge T_f\beta \\
	   &= \int_M f^*(T_{id}\alpha) \wedge f^* (T_{id}\beta) = \int_M f^* (T_{id}\alpha \wedge T_{id}\beta) \\
	   &= deg(f) \cdot \int_N  T_{id}\alpha \wedge T_{id}\beta = deg(f) \int_N (\alpha + d\eta) \wedge (\beta + d\nu)\\
	   &= deg(f) \int_N \alpha \wedge \beta + deg(f) \int_N d(\alpha \wedge \nu) +\\ &+deg(f) \int_N d(\eta \wedge \beta) + deg(f) \int_N d(\eta \wedge d\nu)\\
	   &= deg(f) \int_N \alpha \wedge \beta + 0 = deg(f) \langle [\alpha] , [\beta]\rangle_N.
	    \end{split}
	\end{equation}
We proved point 1. The proof of point 2. is a direct consequence of point 1. and Theorem 4.2 of \cite{Bei3}.
\end{proof}
\subsection{Compact manifolds with unbounded ends}\label{Lott}
Notice that Proposition \ref{proposition1} is coherent with Proposition 5 of Lott in \cite{Lott}. In the work of Lott, it is proved that if two complete Riemannian manifolds $(M,g)$ and $(N,h)$ are isometric outside a compact set, then, for each $k$ in $\mathbb{N}$,
\begin{equation}
    dim(\overline{H}^k_2(M,g)) = +\infty \iff dim(\overline{H}^k_22(N,h)) = +\infty.
\end{equation}
Then, as a consequence of Proposition \ref{proposition1}, we can say something more if we add some assumptions on $f$ on the compact subsets. On the other hand we are not assuming the completeness of $(M,g)$ and $(N,h)$ and we also relax the assumptions on $f$ on the ends.
\begin{cor}
Let $(M,g)$ and $(N,h)$ be two (possibly not complete) oriented Riemannian manifolds. Let $K$ be a compact subset of $M$. We denote by $K' := f(K)$. Assume the existence of a homotopy equivalence $f: (M,g) \longrightarrow (N,h)$ such that $f_{\vert M \setminus K}: (M \setminus K, g) \longrightarrow (N \setminus K', h)$ is a quasi-isometry. Then both the minimal and maximal $L^2$-cohomology groups are isomorphic. The same also happens for the reduced $L^2$-cohomology groups.
\end{cor}
\begin{proof}
Thanks to Proposition \ref{blu} we can assume that $f_{\vert K}$ is an isometry. Fix a number $r > 0$ and denote by $B_r(K)$ (resp. $B_{r}(K')$) the subset of the points whose distance from $K$ (resp. $K'$) is less to $r$. Let $\delta$ be a number small enough such that $B_\delta(K)$ and $B_\delta(K')$ are two tubular neighborhoods of $K$ and $K'$. Then $B_{\delta}(K)$ and $B_{\delta}(K')$ are two open of bounded geometry and $f$ is uniform homotopy equivalence isometric on the unbounded ends. Then we conclude by applying Proposition \ref{proposition1}.
\end{proof}
\begin{rem}
Probably, this result for unreduced $L^2$-cohomology could be also achieved by using the results of \cite{Spes} and a Mayer-Vietoris argument. However, the difficult part of this approach would be the reduced case. Indeed it is a well-known fact that, in general, there is no Mayer-Vietoris sequence for the reduced $L^2$-cohomology.
\end{rem}
\subsection{Manifolds not quasi-isometric with same cohomologies}
We conclude this paper by showing an easy example of two metrics on $\mathbb{R}$ which are not quasi-isometric but they have the same minimal and maximal $L^2$-cohomology groups, both in their reduced and unreduced versions.
\\Let $\phi: \mathbb{R} \longrightarrow \mathbb{R}_{>0}$ be a smooth map such that
\begin{itemize}
    \item $\phi(0) = 1$ and, for each $k$ in $\mathbb{N}$, the derivative of order $k$ in $0$ of $\phi$ is $0$,
    \item outside $[-4, 4]$ we have $\phi(t) = e^{t^2}$.
\end{itemize}
So denote by $g_1$  the Riemannian metric on $\mathbb{R}$ as $g_{1}(t) := \phi(t) dt^2$. Moreover let $\psi: \mathbb{R} \longrightarrow \mathbb{R}$ be the function defined as follows:
\begin{equation}\psi(t) :=
\begin{cases}
 \phi(t - sign(t)1) \mbox{   if   } \vert t \vert > 1\\
 1 \mbox{ otherwise.}
\end{cases}
 \end{equation}
Denote by $g_2$ the Riemannian metric defined as $g_{2}(t) := \psi(t) dt^2$. Notice that $\lim\limits_{t \rightarrow \pm \infty} \frac{\phi(t)}{\psi(t)} = +\infty.$ Then $g_1$ and $g_2$ are not quasi-isometric.
However, if we define the map $f:(\mathbb{R}, g_2) \longrightarrow (\mathbb{R}, g_1)$ as
\begin{equation}
f(t) := \begin{cases}
 t - sign(t)1 \mbox{  if   } \vert t \vert > 1\\
 0 \mbox{ otherwise.}
\end{cases}    
\end{equation}
this is a uniform map isometric on the unbounded ends $(-\infty, -5] \cup [5, +\infty)$ and it is also a uniform homotopy equivalence. So, as a consequence of Subsection \ref{Lott}, we conclude that minimal and maximal $L^2$-cohomology groups, both in their reduced and unreduced versions, are isomorphic.

	\phantomsection
	\bibliographystyle{sapthesis} 
	\end{document}